\documentclass[11pt,a4paper]{article}
\usepackage{style}

\author{Thomas \textsc{Budzinski} and Thomas \textsc{Lehéricy}}
\title{Recurrence of the Uniform Infinite Half-Plane Map via duality of resistances}

\DeclareSymbolFont{calletters}{OMS}{cmsy}{m}{n}
\DeclareSymbolFontAlphabet{\mathcal}{calletters}

\newcommand{\uhpqs}{Q_\infty}

\newcommand{\uhpms}{M_\infty}
\newcommand{\uhpmg}{M^g_\infty}
\newcommand{\Hull}{B^\bullet}

\newcommand{\beads}{beads\xspace}

\newcommand{\eff}{\operatorname{eff}}

\newcommand{\monoQ}{\Q}
\newcommand{\monoQQ}{\QQ}
\newcommand{\biQ}{\Q^{\bullet}}
\newcommand{\biQQ}{\QQ^{\bullet}}

\newcommand{\triQ}{\Q^{\bullet\bullet}}

\newcommand{\triQQ}{\QQ^{\bullet\bullet}}
\newcommand{\triQA}{\A}
\newcommand{\triQQA}{\AA}

\newcommand{\monoM}{\M}
\newcommand{\monoMM}{\MM}
\newcommand{\biM}{\M^\bullet}
\newcommand{\biMM}{\MM^\bullet}
\newcommand{\triM}{\M^{\bullet\bullet}}
\newcommand{\triMM}{\MM^{\bullet\bullet}}

\newcommand{\monoP}{\J}
\newcommand{\monoPP}{\mathcal{J}}
\newcommand{\biP}{{\J^\bullet}}
\newcommand{\biPP}{{\mathcal{J}^\bullet}}

\newcommand{\Core}{\operatorname{TCore}}
\newcommand{\Coreq}{\operatorname{Core}}

\newcommand{\Tutte}{Tutte mapping\xspace}
\newcommand{\Ti}[1]{\TT\l(#1\r)}
\newcommand{\Tti}[1]{\widetilde{\TT}\l(#1\r)}

\newcommand{\UIHPM}{\textsc{uihpm}\xspace}
\newcommand{\UIHPQ}{\textsc{uihpq}\xspace}
\newcommand{\UIHPT}{\textsc{uihpt}\xspace}
\newcommand{\UIPQ}{\textsc{uipq}\xspace}
\newcommand{\UIPT}{\textsc{uipt}\xspace}
\newcommand{\UIPM}{\textsc{uipm}\xspace}

\newcommand{\dloc}{d_{\text{loc}}}

\begin{document}

\maketitle

\begin{abstract}
We study the simple random walk on the Uniform Infinite Half-Plane Map, which is the local limit of critical Boltzmann planar maps with a large and simple boundary. We prove that the simple random walk is recurrent, and that the resistance between the root and the boundary of the hull of radius $r$ is at least of order $\log r$. This resistance bound is expected to be sharp, and is better than those following from previous proofs of recurrence for non bounded-degree planar maps models. Our main tools are the self-duality of uniform planar maps, a classical lemma about duality of resistances and some peeling estimates. The proof shares some ideas with Russo--Seymour--Welsh theory in percolation.
\end{abstract}

\section*{Introduction}

\paragraph{Local limits of random planar maps.}

Local limits of random planar maps have been the subject of extensive research in the last fifteen years. The first local limit that was introduced is the Uniform Infinite Planar Triangulation of Angel and Schramm \cite{AS03}. Since then, many other models with various topologies and local structures have been investigated. Such models proved to be the local limits of finite planar maps chosen uniformly in various classes when the size goes to infinity. Examples of such models with the topology of the plane include the \UIPQ \cite{Kri05, CD06} for quadrangulations, the \UIPM \cite{MN13} for general maps or the wide family of infinite Boltzmann planar maps \cite{BS13,St18}. On the other hand, local limits of models with a boundary when both the boundary length and the total volume go to infinity have also been investigated, and are usually maps with the topology of the half-plane. Such maps include the \UIHPT \cite{Ang05} for triangulations and the \UIHPQ with a general boundary \cite{CMboundary} or with a simple boundary \cite{CCUIHPQ} for quadrangulations.

\paragraph{The simple random walk on random planar maps.}
Many features of these local limits are now quite well understood such as their volume growth or percolation critical points (see e.g. \cite{CPeccot} for a survey). However, some aspects of the simple random walk on random planar maps remain quite mysterious. One of the most notable results about these random walks is the proof of the recurrence of a wide class of models including the \UIPT, \UIPQ and \UIPM by Gurel-Gurevich and Nachmias \cite{GGN12}, via the theory of circle packings (see also Lee \cite{Lee17} for a different proof). About half-plane models, the recurrence of the \UIHPT has also been established by Angel and Ray \cite{AR18}, also relying on circle packings. On a different note, the displacement exponent of the simple random walk on the \UIPT has been computed recently by Gwynne and Hutchcroft \cite{GH18}, as well as polylogarithmic upper bounds for resistances by Gwynne and Miller \cite{GM17}. These two results have only been proved for triangulations, but hold in the more general context of infinite triangulations carrying certain statistical physics models. However, they are up to polylogarithmic factors, which means that the precise order of magnitude of many quantities of interest is still open, such as the resistance between the root vertex and the boundary of the ball of radius $r$, or the displacement of the walk during time $n$.

\paragraph{Uniform infinite maps.}
The model we are interested in in this work is the \emph{Uniform Infinite Half-Plane Map} with a simple boundary (\UIHPM), that we denote by $\uhpms$. We note right now that its full plane analog, the \UIPM, has already been defined by M\'enard and Nolin in \cite{MN13}. A natural way to study the \UIPM is to use its relation with the \UIPQ via the classical Tutte bijection between general maps and quadrangulations. In particular, metric properties of the \UIPM are studied in this way in \cite{Leh18}. We will give a precise definition of the \UIHPM in Theorem \ref{thm_defn_UIHPM} below. As we will see, it can also be constructed by a version of the Tutte bijection, but the presence of a boundary makes the definition of the model and the proof of a local convergence result more complicated.

\paragraph{Recurrence of the UIHPM and resistances.}
Let $\rho$ be the root vertex of $\uhpms$. In a planar (or half-plane) map $m$, the \emph{ball} $B_r(m)$ of radius $r \geq 1$ is the set of all the vertices at distance at most $r$ from the root vertex $\rho$.
The \emph{hull} $\Hull_r(m)$ of radius $r$ is the union of the ball of radius $r$ and of all the bounded connected components of its complementary. Finally, in a graph $G$, we denote by $\Reff ^{G} \l( x \leftrightarrow y \r)$ the effective electric resistance between two vertices $x$ and $y$ in $G$, and by $\Reff ^{G} \l( A \leftrightarrow B \r)$ the effective resistance between two sets of vertices $A$ and $B$.
\begin{theorem}\label{main_thm}
There is a constant $c>0$ such that almost surely, for $r$ large enough, we have
\[ \Reff ^{\uhpms} \l( \rho \leftrightarrow  \uhpms \setminus \Hull_r(\uhpms) \r) \geq c \log r. \]
In particular, the map $\uhpms$ is almost surely recurrent.
\end{theorem}
Note that the first part implies that the resistance goes to $+\infty$ as $r \to +\infty$, so the second part is an immediate consequence of the first one. Moreover, the order $\log r$ of our lower bound is believed to be sharp for many models. To our knowledge, this is the first time that a logarithmic lower bound is established for a non-bounded-degree model. The recurrence proofs of \cite{GGN12} or \cite{AR18} relying on the theory of circle packings are only expected to give a lower bound of order $\log \log r$ if the degrees are unbounded. The cost to pay for this improved lower bound is a lack of universality, as the tools we use are very specific to self-dual models, which for example do not include the \UIPT.

\paragraph{Construction of the UIHPM.}
Since the \UIHPM $M_{\infty}$ has not yet been defined in the literature, let us now be more precise about what it is. We denote by $M_p$ a random map with a simple boundary of length $p$ and with critical Boltzmann distribution (i.e. the probability of sampling a given map with $e$ edges in total is proportional to $12^{-e}$). We also denote by $Q_{\infty}$ the Uniform Infinite Half-Plane Quadrangulation (\UIHPQ) with a simple boundary \cite{CCUIHPQ} (see Section \ref{Sec_UIHPM} for precise references and definitions of these objects). Finally, we recall that every planar quadrangulation is bipartite, so its vertices can be coloured in black and white in such a way that two neighbours always have opposite colours, and the root vertex is white.
A straightforward generalization of the Tutte bijection, which we call the \emph{\Tutte} since it is not bijective anymore (see Section \ref{subsec_Tutte_UIHPM} for details), associates to any quadrangulation $q$ with a simple boundary a map $\Ti{q}$ with a boundary, whose vertices are the white vertices of $q$, and whose edges are the white-white diagonals of the internal faces of $q$. The next result will be our definition of the \UIHPM.

\begin{theorem}\label{thm_defn_UIHPM}
We have the convergence in distribution
\[M_p \xrightarrow[p\to+\infty]{(d)} M_{\infty}\]
for the local topology, where $M_{\infty}$ is an infinite half-plane map with a simple boundary that we call the \UIHPM. Moreover, the map $M_{\infty}$ can be obtained from $Q_{\infty}$ by applying the \Tutte and pruning the boundary to make it simple.
\end{theorem}
The pruning procedure is necessary because, as we will see in Section \ref{subsec_Tutte_UIHPM}, the image of $Q_{\infty}$ by the \Tutte does not have a simple boundary. The local limit statement justifies the interest of the model. On the other hand, the proof of Theorem \ref{main_thm} relies on the construction of $M_{\infty}$ using the \Tutte.

\paragraph{Ingredients of the proof of Theorem \ref{main_thm}.}
Let $\uhpmg$ be the image of $\uhpqs$ by the \Tutte (it is a half-plane map with a non-simple boundary). We will prove Theorem \ref{main_thm} for $\uhpmg$, and then transfer the resistance estimates from $\uhpmg$ to $\uhpms$ using the boundary pruning procedure. Our proof of Theorem \ref{main_thm} for $\uhpmg$ relies mostly on two ingredients:
\begin{itemize}
\item
the first one is the self-duality of the model. By combining this property with a classical result on the duality of resistances in planar maps, we obtain an estimate on the resistance between the “bottom” and the “top” parts of the boundary in a certain self-dual block. In particular, this resistance has a probability bounded from below to be at least $1$.
\item
The second ingredient consists of some peeling estimates on the \UIHPQ $Q_{\infty}$, which is naturally coupled with $\uhpmg$ via the \Tutte. These estimates show that, between $\rho$ and the complementary of the hull of radius $r$ in $\uhpmg$, we can find roughly $\log r$ disjoint blocks of the kind described above, so the resistance must be at least of order $\log r$.
\end{itemize}

\paragraph{An analogy with percolation.}
Finally, let us note that somewhat similar ideas have been used in a different context by Biskup, Ding and Goswami in \cite{BDG16}. The goal of \cite{BDG16} is to study the simple random walk on $\Z^2$ with random conductances given by the exponential of the Gaussian free field. The main idea of the proof of \cite{BDG16} is to start from the self-duality of a square (which plays the same role as our blocks), and then to use Russo--Seymour--Welsh-type ideas to extend estimates to more complicated domains. The RSW analogy is also of interest in our case: the use of logarithmically many blocks to separate the root from infinity is reminiscent of the proof of polynomial decay of one-arm events in percolation. We also believe that RSW-type ideas would be needed to obtain upper bounds on the resistances using similar arguments to ours. 

\paragraph{Structure of the paper.}
The structure of the paper is as follows. In Section \ref{Sec_UIHPM}, we define the objects involved in this work, and in particular the \UIHPM $\uhpms$. Note that we first define $\uhpms$ as the map built from the \UIHPQ by the \Tutte, followed by a boundary pruning procedure. In Section \ref{Sec_selfduality}, we prove a resistance estimate on a self-dual block. Section \ref{Sec_peeling} is devoted to peeling estimates allowing us to find one suitable self-dual block at some fixed scale in $\uhpms$. In Section \ref{Sec_proof_main_theorem}, we conclude the proof of Theorem \ref{main_thm} by applying repeatedly the results of Sections \ref{Sec_selfduality} and \ref{Sec_peeling} at different scales. In Section \ref{Sec_UIHPM_is_limit_of_Boltzmann_maps}, we prove Theorem \ref{thm_defn_UIHPM}, i.e. that the map $\uhpms$ built in Section 1 is the local limit of critical Boltzmann maps with simple boundary when the perimeter goes to infinity. We highlight that Section \ref{Sec_UIHPM_is_limit_of_Boltzmann_maps} is completely independent of the rest of this work. Finally, in Section \ref{Sec_perspectives} we make a few comments about the robustness of the argument and quickly discuss some perspectives.

\paragraph{Acknowledgments.} The authors acknowledge the support of ERC Advanced Grant 740943 GeoBrown. The first author thanks Nicolas Curien and Asaf Nachmias for useful discussions.

\tableofcontents

\section{Definitions and construction of the model}
\label{Sec_UIHPM}

\subsection{Basic definitions and combinatorics}
\label{subsec_basics}

\paragraph{Finite and infinite maps with a boundary.}
A \emph{planar map} (or just \emph{map}, since we only work with planar objects) is a planar graph embedded in the sphere, considered up to orientation-preserving homeomorphisms. We always consider \emph{rooted} maps, which come with a distinguished oriented edge called the \emph{root edge}. The starting vertex of the root edge is the \emph{root vertex}, and the face incident to the right of the root edge is the \emph{external face}. If $m$ is a map, we denote the boundary of its external face by $\partial m$; the boundary is \emph{simple} if $\partial m$ is a simple cycle, and \emph{generic} otherwise. A \emph{finite quadrangulation with a boundary} is a finite map in which all the non-external faces have degree $4$. We also stick to the convention that there is one “trivial” quadrangulation $\dagger$ of perimeter $2$ with $0$ internal face, consisting of two vertices and a single edge. Note that quadrangulations with a boundary are always bipartite. Therefore, we always assume that their vertices are coloured in white and black in such a way that two adjacent vertices always have different colors, and that the root edge is directed from a white vertex to a black one. By convention, we draw maps in the plane in such a way that their external face is the unbounded face.

We also consider infinite maps. All the infinite maps we consider have one end, and we draw them in the plane so that faces with finite degree are sent to compact subsets of the plane. A \emph{half-plane} map is an infinite map in which all the faces have finite degrees, except the external face, which is infinite. In particular, a \emph{half-plane quadrangulation} is a half-plane map in which all the finite faces have degree $4$. We usually represent half-plane maps with a simple boundary in such a way that the boundary is a horizontal line, and the external face is below it.

\paragraph{Local topology.}
If $m$ is a map, we denote by $d_m$ the graph distance on the set of vertices of $m$. We will be interested in \emph{local convergence} of planar maps. If $m$ is a map (or quadrangulation) with a boundary and $r\geq 0$, we denote by $[m]_r$ the map formed by the internal faces of $m$ incident to only vertices at graph distance at most $r$ from the root, plus the knowledge of which of the vertices and edges belong to $\partial m$\footnote{If $r$ is too small and $[m]_r$ does not contain any internal face, it is simply the trivial map with $1$ vertex and $0$ edge.}. If $m$ and $m'$ are two (finite or infinite) maps, then we write
\[\dloc(m,m')=\l( 1+\sup \{ r \geq 0 \ | \ \mbox{$[m]_r$ and $[m']_r$ are isomorphic} \} \r)^{-1}.\]
Then $\dloc$ is a distance called the \emph{local distance}, which makes the space of finite and infinite maps with a boundary a complete, separable space. Note that, while our definition of $[m]_r$ is not the most common definition of a ball, the local topology does not really depend on the definition of a ball that we pick, and this one will be more convenient for us later.

\paragraph{Combinatorics.}
Let us now recall a few formulas about the combinatorics of quadrangulations with a boundary. In all this work, we use lower case letters such as $q$ for deterministic maps, upper case letters like $Q$ for random maps, blackboard bold letters like $\Q$ for sets of maps, and calligraphic letters like $\QQ$ for generating functions. For every $p \geq 1$ and $n \geq 0$, let $\Q_{n,p}$ be the set of quadrangulations with a simple boundary of length $2p$ and $n$ internal faces. Then exact formulas for $\# \Q_{n,p}$ can be found in \cite{BG09} and give
\[ \# \Q_{n,p} \underset{n \to +\infty}{\sim} c_p \ 12^n \ n^{-5/2}. \]
In particular, for $x \geq 0$, the sum $\sum_{n \geq 0} \# \Q_{n,p} \times x^n$ is finite if and only if $x \leq 1/12$. For $p \geq 1$, we define
\begin{equation}\label{formula_quad_Zp}
\QQ_p \eqdef \sum_{n \geq 0} \# \Q_{n,p} \l( \frac{1}{12}\r)^n=12 \l( \frac{2}{3} \r)^p \frac{(3p-4)!}{(p-2)!(2p)!} \underset{p \to +\infty}{\sim} \frac{2}{9\sqrt {3\pi}} \times \l( \frac{9}{2} \r)^p \times p^{-5/2},
\end{equation}
where the formula is a consequence of the results of \cite{BG09}, and is given in a close form\footnote{We do not have exactly the same formula because in \cite{ACpercopeel} $n$ counts the internal vertices, whereas here $n$ counts the internal faces.} in \cite{ACpercopeel}. We also define a \emph{critical Boltzmann quadrangulation with a simple boundary of length $2p$}, denoted by $Q_p$, as a random quadrangulation such that, for every $n\geq 0$ and $q \in \Q_{n,p}$, we have $\P \l( Q_p=q \r)=\frac{1}{\QQ_p} \l( \frac{1}{12} \r)^n$.

We also need some combinatorial formulas about general maps with a simple boundary. For every $p \geq 1$, let $\monoM_{n,p}$ be the set of maps with a simple boundary of length $p$ and $n$ edges in total. These are in bijection via the \Tutte with a class of quadrangulations with boundary called \emph{quadrangulations with truncated boundary} of perimeter $2p$ and $n$ internal faces, which are enumerated in \cite{Kri05}. We describe the \Tutte carefully in Section \ref{subsec_Tutte_UIHPM}, but let us give right now consequences on the enumeration of $\monoM_{n,p}$ (the results that follow are stated in \cite{gall2017separating} for quadrangulations with truncated boundary). For any $p \geq 1$, we have
\[ \# \monoM_{n,p} \underset{n \to +\infty}{\sim} c'_p 12^n n^{-5/2} \]
for some constants $c'_p$ for $p\geq 1$. 
In particular, we can define 
\[ \MM_p \eqdef \sum_{n \geq 0} \# \monoM_{n,p} \l( \frac{1}{12} \r)^n, \]
and
\[ \MM \l( \frac{1}{12}, z \r) \eqdef \sum_{n \geq 0, \, p \geq 1} \# \monoM_{n,p} \l( \frac{1}{12} \r)^n z^p  = \sum_{p \geq 1} \MM_p z^p. \]
No explicit formula for $\MM_p$ is known, but we have
\begin{equation}\label{generating_function_M}
\MM \l( \frac{1}{12}, z \r) = \frac{1}{24} \sqrt{(18-z)(2-z)^3}-\frac{1}{2}+\frac{z}{2}-\frac{z^2}{24}.
\end{equation}
By standard singularity analysis (see e.g. \cite{FS09}), we deduce
\begin{equation}\label{asymptotic coefficients_M}
\MM_p \underset{n \to +\infty}{\sim} \frac{1}{2 \sqrt{2 \pi}} \times 2^p \times p^{-5/2}.
\end{equation}

\subsection{The Uniform Infinite Half-Plane Quadrangulation}\label{subsec_UIHPQ}

\paragraph{The UIHPQ.}
We now define and review basic properties of the Uniform Infinite Half-Plane Quadrangulation with a simple boundary (\UIHPQ), which is a natural model of random half-plane quadrangulation. The \UIHPQ, denoted by $Q_{\infty}$, is defined as the local limit in distribution as $p \to +\infty$ of the critical Boltzmann quadrangulations $Q_p$. This convergence is stated in \cite{ACpercopeel}, and can be proved in a similar way as its analog for triangulations \cite{Ang05}. See also \cite{CMboundary} and \cite{CCUIHPQ} for two other constructions of the \UIHPQ (with non-simple boundary) based on bijections with labelled trees (the \UIHPQ with simple boundary is then obtained by a pruning procedure). Moreover, the \UIHPQ is a.s. a \emph{one-ended} map, i.e. the complement of any finite set of vertices has only one infinite connected component.

In order to precisely describe the distribution of the \UIHPQ, we need to define a notion of \emph{sub-map} of a half-plane quadrangulation. 
Let $q_\infty$ be an infinite half-plane quadrangulation, and let $q$ be a finite map as in Figure \ref{fig_submap_quadrangulation}. 
A formal definition would be that $q$ is a finite quadrangulation with a simple boundary, where the boundary vertices and edges of $q$ bear labels indicating if and where they are glued to the infinite boundary of $q_\infty$, and at least one boundary edge of $q$ is glued on the boundary of $q_\infty$. We write $q \subset q_{\infty}$ and say that $q$ is a sub-map of $q_\infty$ if $q_{\infty}$ can be obtained from $q$ by filling all the finite holes with finite quadrangulations with a (simple) boundary, and the infinite hole with a half-plane quadrangulation with a simple boundary.

\begin{figure}
\begin{center}
\includegraphics[width=0.6\textwidth]{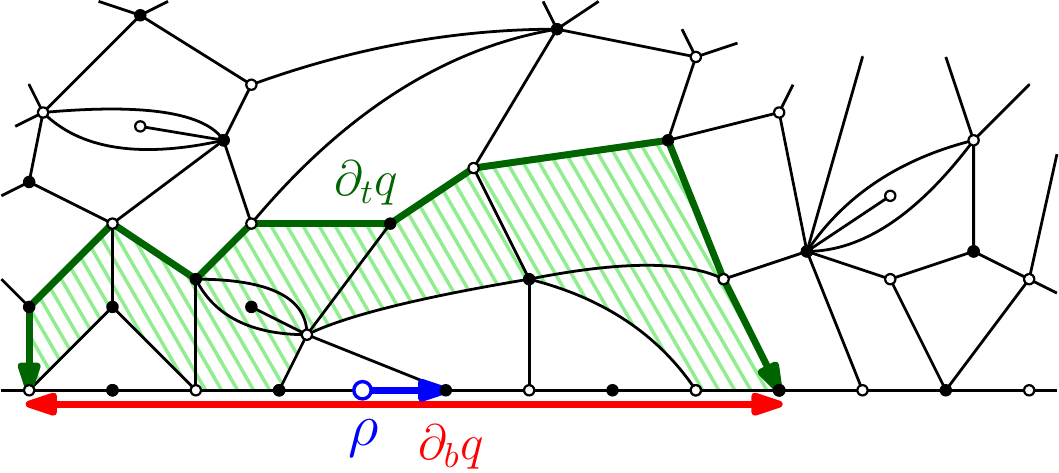}
\end{center}
\caption{A half-plane quadrangulation $q_{\infty}$ and a finite sub-map $q \subset q_{\infty}$ (in hatched green). The sub-map $q$ has 8 internal faces (i.e. $|q|=8$) and 2 finite holes. 
 Here we have $|\partial_b q|=|\partial_t q|=9$. There are two holes $h_1$ and $h_2$ with perimeters $4$ and $8$. Note that the root vertex has to be on $\partial_b q$, so if the root was far away on the right, we would need to add a segment on the right to both $\partial_b q$ and $\partial_t q$.}
\label{fig_submap_quadrangulation}
\end{figure}

If $q \subset q_{\infty}$ is a sub-map of a half-plane quadrangulation, we denote by $|q|$ the number of internal faces of $q$. We note that $q$ has a unique infinite hole $h_{\infty}$. The \emph{bottom boundary} of $q$, denoted by $\partial_b q$, is the largest segment of $\partial q_{\infty}$ containing the root vertex of $q_{\infty}$ and all the vertices that belong both to $\partial q_{\infty}$ and to $\partial q$ (see Figure \ref{fig_submap_quadrangulation}). The \emph{top boundary} of $q$, denoted by $\partial_t q$, is the segment of $\partial h_{\infty}$ between the two ends of the bottom boundary. It follows from this definition and the formula \eqref{formula_quad_Zp} that for any sub-map $q$ with no finite hole, we have
\begin{equation}\label{formula_distrib_UIHPQ}
\P \l( q \subset Q_{\infty} \r)=\l( \frac{1}{12} \r)^{|q|} \times \l( \frac{9}{2} \r)^{\frac{|\partial_t q|-|\partial_b q|}{2}},
\end{equation}
which characterizes the distribution of $Q_{\infty}$. Note that, although sub-maps with no finite hole are enough to describe the distribution of $Q_{\infty}$, it will be important later in this work to consider sub-maps with one or two finite holes.

\paragraph{Peeling the UIHPQ.}

An interesting consequence of \eqref{formula_distrib_UIHPQ}, also called the \emph{spatial Markov property}, is that it allows to explore $Q_{\infty}$ in a Markovian way via \emph{peeling explorations}. A \emph{peeling exploration} is an increasing sequence of sub-maps $(Q^i)_{i \geq 0}$ of $Q_{\infty}$, whose choice depends on a \emph{peeling algorithm}. A peeling algorithm $\AA$ is a way to assign to every possible sub-map $q$ an edge $\AA(q)$ on the boundary of the infinite hole of $q$, in a way that only depends on $q$. The edge $\AA(q)$ may be either incident to $q$, or to the boundary of the half-plane. The sequence $Q^i$ is then defined as follows: $Q^0$ is the empty map and, for every $i \geq 1$, the sub-map $Q^{i}$ is obtained from $Q^{i-1}$ by adding the unexplored face $F^{i}$ incident to $\AA(Q^{i-1})$ and by filling all the finite holes. Such explorations are called \emph{filled-in}. The peeling that we consider is a peeling \emph{with simple boundary}, i.e. when we explore a new face, we learn right now which of its vertices belong to the previous boundary, so that the map filling the infinite hole always has a simple boundary. This peeling for quadrangulations has first been studied in \cite{ACpercopeel} (see also \cite{R15} and \cite[Section 6.2]{CLGpeeling}). It is different from the \emph{peeling with general boundary} introduced in \cite{Bud15} and studied in \cite{CPeccot} in a systematic way.

At each peeling step, the peeled quadrangle may take different forms, which are described on Figure \ref{fig_peeling_cases}. The probability of each of these cases can be computed explicitly, see e.g. \cite{ACpercopeel} or \cite{R15}. However, we do not need the exact formulas, and we state right now the computations that will be useful for us. First, if $X_i=|\partial_t Q^i|-|\partial_b Q^i|$, then $X$ is a centered random walk on $\Z$, which lies in the domain of attraction of a completely asymmetric $3/2$-stable law. Therefore, it converges to a $3/2$-stable Lévy process with no positive jumps (see \cite[Section 6.2]{CLGpeeling} for more details). Second, the probability that the peeled face forms two finite holes with perimeters $2\ell_1$ and $2\ell_2$ (which corresponds to the three bottom cases of Figure \ref{fig_peeling_cases}) is equal to
\begin{equation}\label{eqn_bad_peeling_case}
\frac{1}{4} \l( \frac{2}{9} \r)^{\ell_1+\ell_2} \monoQQ_{\ell_1} \monoQQ_{\ell_2} \leq C \ell_1^{-5/2} \ell_2^{-5/2},
\end{equation}
where the $\QQ_\ell$ were introduced in \eqref{formula_quad_Zp} and $C$ is an absolute constant. The factor $\frac{1}{4}$ consists of a factor $\frac{1}{12}$ for the Boltzmann weight on the peeled face, and a factor $3$ for the three cases in the bottom of Figure \ref{fig_peeling_cases}. Finally, in every case, conditionally on the peeling case and the different boundary lengths which appear, the quadrangulations filling the finite holes are critical Boltzmann quadrangulations with a simple boundary of fixed length, and the quadrangulation filling the infinite hole has the same distribution as $Q_{\infty}$. This is known as the \emph{spatial Markov property} of $Q_{\infty}$.

\begin{figure}
\begin{center}
\includegraphics[scale=0.7]{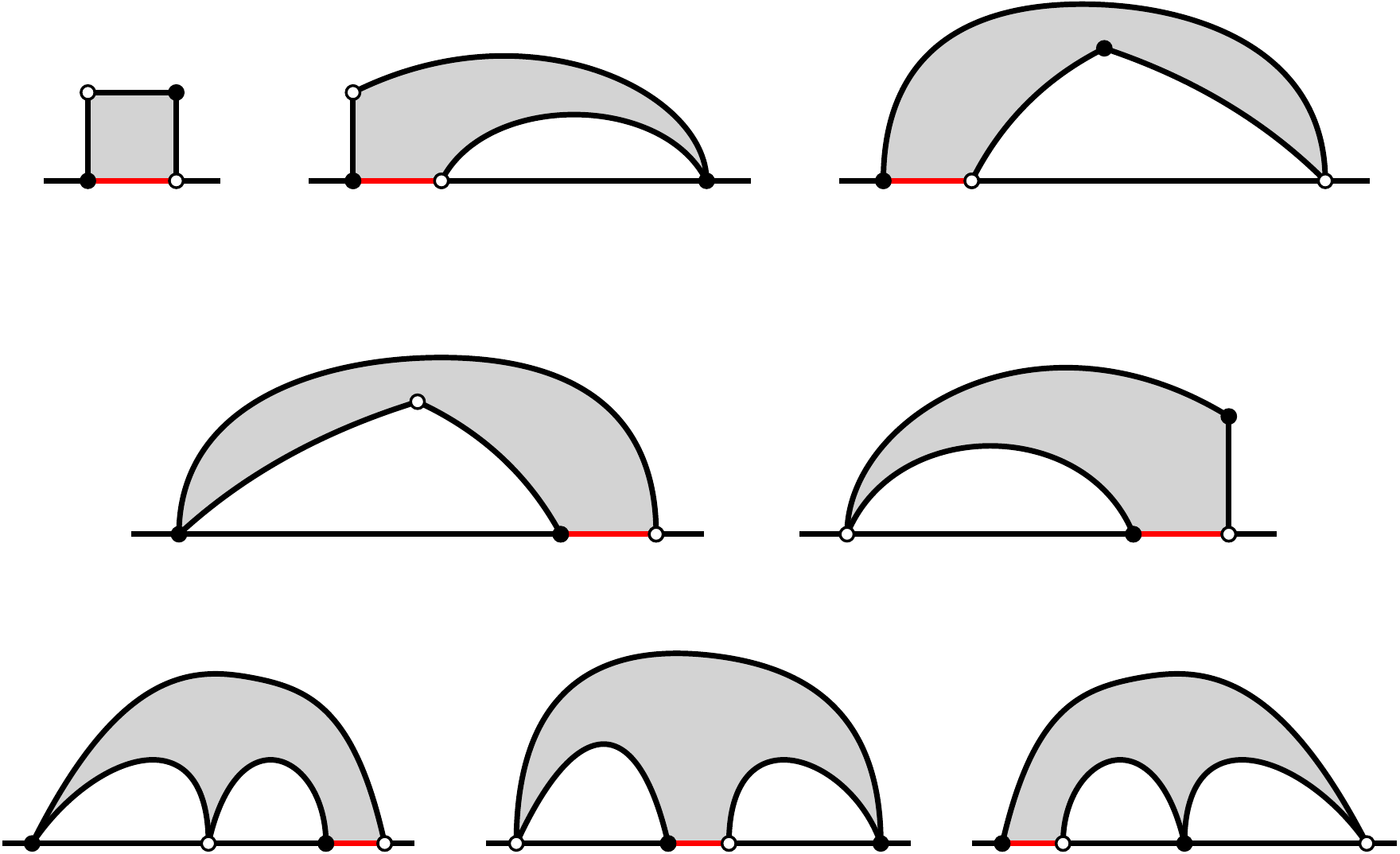}
\end{center}
\caption{The different peeling cases for quadrangulations with simple boundary (if the peeled edge has a white vertex on its left, the possible cases are the symmetric of those shown on the figure).}\label{fig_peeling_cases}
\end{figure}

In Section \ref{Sec_peeling}, we will use a peeling exploration which is stopped at a stopping time $\tau$, and where the holes formed by the peeling step at time $\tau$ are not filled. Hence, the maps $Q^i$ for $0 \leq i < \tau$ will have no finite hole, whereas $Q^{\tau}$ may have one or two finite holes.

\subsection{The \Tutte{} and the UIHPM}\label{subsec_Tutte_UIHPM}

The goal of this subsection is to construct the Uniform Infinite Half-Plane Map with a simple boundary (\UIHPM). We build it from the \UIHPQ $Q_{\infty}$, and will prove later in Section \ref{Sec_UIHPM_is_limit_of_Boltzmann_maps} that it is the local limit as $p \to +\infty$ of critical Boltzmann maps with a simple boundary of length $p$.

\paragraph{The \Tutte.}
Our construction relies on the well-known \emph{Tutte bijection} between finite quadrangulations and finite maps. Let us first briefly recall the bijection in the case where there is no boundary. Consider a rooted quadrangulation $q$ with $n$ faces, and color its vertices so that the root vertex is white and any pair of adjacent vertices have opposite colors (we can do this since quadrangulations are bipartite). The image of $q$ by the Tutte bijection is the map with $n$ edges obtained by drawing a diagonal between white corners in each face of $q$ and keeping only the white vertices of $q$ and the added diagonals. We root this map at the diagonal drawn in the face of $q$ that is incident to the left of the root edge of $q$, so that the root vertex is the same in $q$ and in its image.

We extend this bijection to quadrangulations with a simple boundary, at the cost of losing the bijectivity property. Consider a quadrangulation $q$ with a simple (finite or infinite) boundary. In every internal face of $q$, draw a diagonal between the two white corners. Remove all black vertices of $q$ and edges of $q$, keeping only white vertices and the added diagonals, and root the obtained map as before. This gives a connected (this follows from the fact that $q$ has a simple boundary) rooted map $\Ti{q}$ with a boundary, which is in general not simple (see Figure \ref{Fig image by Tutte}). Note that, by our choice of root, the root face of $\Ti{q}$ contains the root face of $q$.

\begin{figure}[h!]
\begin{center}
\includegraphics[width=0.4\textwidth]{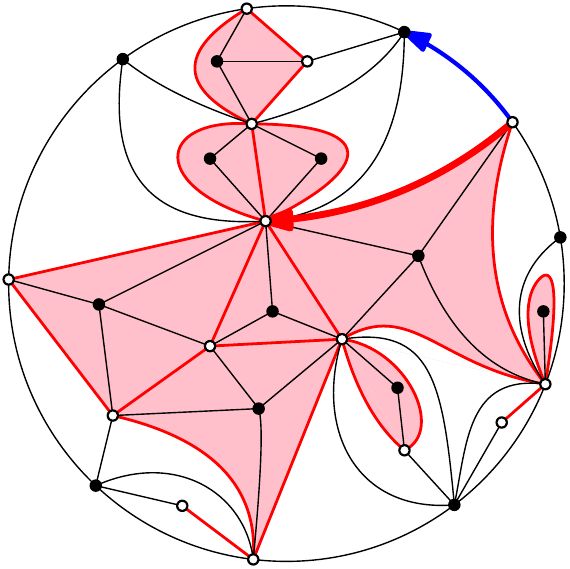}\qquad\qquad
\includegraphics[width=0.4\textwidth]{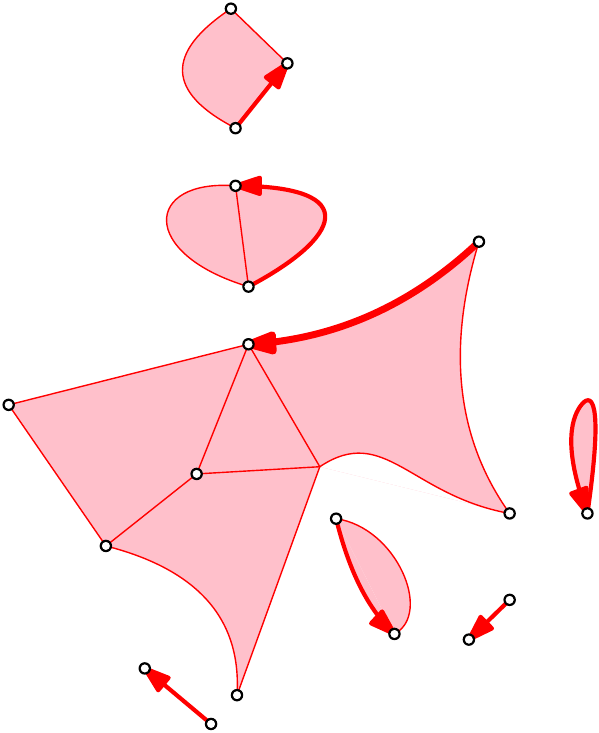}
\end{center}

\caption{The image of a quadrangulation with simple boundary $q$ by the \Tutte does not in general have a simple boundary. On the left, $q$ is drawn in black, with its root edge in blue; $\Ti{q}$ (in red) has its root edge in bold red. On the right, cutting $\Ti q$ at the pinchpoints of its boundary gives the decomposition of $\Ti{q}$ into \beads. Each bead is a map with simple boundary, rooted at their vertex closest to the root vertex of $\Ti{q}$.}
\label{Fig image by Tutte}
\end{figure}
 
Note also that this procedure is not a bijection: the two quadrangulations in \figref{Fig contreexemple Tutte is a bijection} have the same image. Therefore, we will call the function $q \to \Ti{q}$ the \Tutte.

\begin{figure}[h!]
\begin{center}
\includegraphics[width=0.3\textwidth]{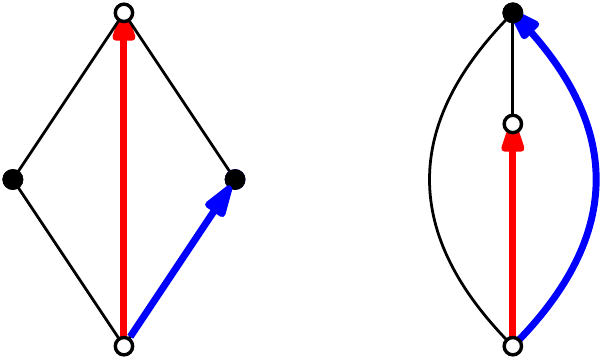}
\end{center}
\caption{The \Tutte is not injective: two quadrangulations that have the same image (in red) by the \Tutte.}
\label{Fig contreexemple Tutte is a bijection}
\end{figure}

Since we want to build a model with a simple boundary, a natural idea is to consider the “core with simple boundary” of the map $\Ti{Q_{\infty}}$. Before formalizing this idea, let us explain which are the quadrangulations sent to maps with a simple boundary by the \Tutte.
This will be useful in Section \ref{Sec_UIHPM_is_limit_of_Boltzmann_maps}. A \emph{quadrangulation with truncated boundary} is a quadrangulation with a simple boundary where all black vertices on the boundary have degree $2$ (see the left part of \figref{Fig super-simple boundary}). These quadrangulations were first considered and enumerated in \cite{Kri05} (see also \cite{gall2017separating, Leh18}), where they play a crucial role in the so-called skeleton decomposition of quadrangulations. Note that our definition of a truncated boundary agrees with \cite{Kri05} and is slightly different from the one in \cite{gall2017separating, Leh18}, where the faces incident to the boundary are actually triangles. However, both notions of truncated boundary are in bijection by simply cutting the faces incident to the boundary along their white-white diagonals. It is easy to check that for every $n, p\geq 1$, the \Tutte $\TT$ is in fact a bijection between the set of quadrangulations with truncated boundary of length $2p$ and $n$ inner faces
, and the set of maps with a simple boundary of length $p$ and $n$ edges in total. Therefore, the formulas given by \cite{Kri05} to count quadrangulations with truncated boundary also allow to count maps with a simple boundary, which justifies \eqref{generating_function_M} and \eqref{asymptotic coefficients_M}.

\begin{figure}[h!]
\begin{center}
\includegraphics[width=0.7\textwidth]{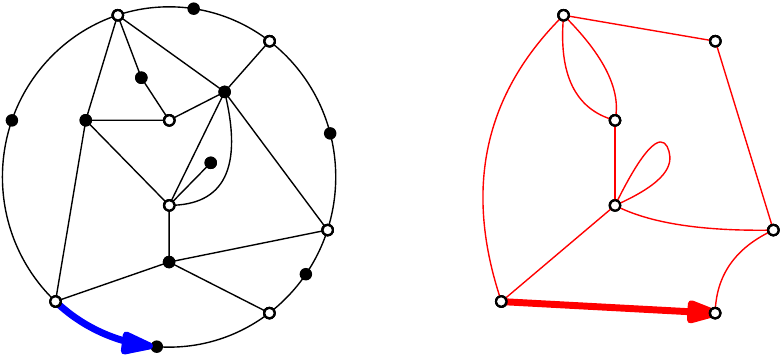}
\end{center}
\caption{The \Tutte $\TT$ is a bijection between the set of quadrangulations with truncated boundary, i.e. quadrangulations with truncated boundary and the set of maps with simple boundary. On the left, a quadrangulation $q$ with truncated boundary of length 10 and 11 inner faces, and on the right its image $\Ti{q}$.}
\label{Fig super-simple boundary}
\end{figure}

\paragraph{The UIHPM.}

If $m$ is a map, we call a vertex on $\partial m$ a \emph{pinchpoint} if at least two of its corners are incident to the external face. Cutting $m$ at its pinchpoints (see the right of \figref{Fig image by Tutte}) yields a collection of \emph{beads}, where each bead is a map with simple boundary. We root the beads at the vertex of their boundary that is the closest to the root vertex of $m$, as on \figref{Fig image by Tutte}, so that the external face of the bead contains the external face of $m$.

\begin{lemma}
\label{Lemma_exist_uniq_core}
Almost surely $\Ti{Q_\infty}$ has a unique infinite bead. This is a half-plane map with a simple boundary, that we call the \UIHPM and denote by $M_{\infty}$.
\end{lemma}

As explained in the Introduction, this will be our working definition of the \UIHPM, and we will prove in Section \ref{Sec_UIHPM_is_limit_of_Boltzmann_maps} (Theorem \ref{thm_defn_UIHPM}) that this coincides with the local limit of Boltzmann maps with a large, simple boundary.

\begin{proof}
We first note that any two beads are separated by two inner edges of $Q_\infty$, joining the same inner white vertex to two black boundary vertices, which may be the same or not, see \figref{Fig image by Tutte}. Therefore, the uniqueness of the infinite bead is an immediate consequence of the one-endedness of $Q_{\infty}$. Hence, it is enough to prove its existence.

To prove there is an infinite bead, it is enough to show that there are only finitely many pinchpoints separating the root from infinity. As can be seen on \figref{Fig image by Tutte}, a pinchpoint in the map $\Ti{q}$ is a white internal vertex of the quadrangulation $q$, which is linked to two black vertices of $\partial q$. If this pinchpoint separates the root vertex from infinity, then one of the black vertices is on the left of the root and the other is on the right. Therefore, to prove that $\Ti{Q_{\infty}}$ has an infinite bead, it is enough to show that there are only finitely many pairs $(u,v)$ of vertices of $\partial Q_{\infty}$, with $u$ on the left of the root and $v$ on its right, such that $d_{Q_{\infty}}(u,v)=2$.

We actually check this in the \UIHPQ \emph{with a general boundary}, that we denote by $Q_{\infty}^g$. This version of the \UIHPQ is defined e.g. in \cite{CMboundary}. Since $Q_{\infty}$ can be built from $Q_{\infty}^g$ by pruning the boundary \cite{CMboundary} in such a way that $\partial Q_{\infty} \subset \partial Q_{\infty}^g$, the claim for $Q_{\infty}^g$ implies the claim for $Q_{\infty}$. For this, let $(c_i)_{i \in \Z}$ be the list of corners of the external face of $Q_{\infty}^g$, and let $u_i^g$ be the vertex incident to $c_i$ (note that $i \to u_i^g$ is not injective since the boundary is not simple). Then we want to check that there are only finitely many pairs $(i,j)$ with $i,j>0$ such that $d_{Q_{\infty}^g}(u_{-i}^g, u_j^g)=2$.

For this, we rely on results from \cite{CCUIHPQ}, which gives a description of $Q_{\infty}^g$ in terms of a labelled forest where the labels correspond to the graph distances to the root vertex. In particular, $\l( d_{Q_{\infty}^g (u_0^g, u_i^g) }\r)_{i \in \Z}$ is a discrete Bessel process, so it is easy to control. More precisely, since $Q_{\infty}^g$ is stationary under root translation, we have
\begin{align*}
\sum_{i,j>0} \P \l( d_{Q_{\infty}^g}(u_{-i}^g, u_j^g) = 2 \r) &= \sum_{i,j>0} \P \l( d_{Q_{\infty}^g}(u_{0}^g, u_{i+j}^g) = 2 \r)\\
&\leq \sum_{i,j>0} C (i+j)^{-5/2}\\
&<+\infty,
\end{align*}
where $C$ is an absolute constant, and we used Lemma 3.5 of \cite{CCUIHPQ} to bound the probabilities. The claim follows by the Borel-Cantelli lemma. This concludes the proof.
\end{proof}

\section{A self-dual block}\label{Sec_selfduality}

The goal of this section is to prove, by using self-duality, that in the finite random maps obtained from critical Boltzmann quadrangulations via the Tutte mapping, a certain resistance is at least $1$ with positive probability. Our goal is Lemma \ref{lem_duality_random} below. These maps will later be used as “building blocks”. Our main tool is the classical Lemma \ref{lem_duality_resistances} about duality of resistances. Let $m$ be a finite planar map drawn in the plane, and let $a_1$ and $a_3$ be two distinct vertices incident to its outer face. We draw two infinite half-lines from $a_1$ and $a_3$ that split the outer face in two faces $a_2^*$ and $a_4^*$. We denote by $m^*$ the dual map whose vertices are $a_2^*$, $a_4^*$ and the internal faces of $m$, and whose edges are the dual edges of those of $m$ (the two infinite half-lines are not edges of $m$). An example can be seen on the right of Figure \ref{fig_duality_mp}. We recall that $\Reff^{m} (a \leftrightarrow b)$ is the electric resistance between two vertices $a$ and $b$ in a map $m$.

\begin{lemma}\label{lem_duality_resistances}
With the above notation, we have
\[ \Reff^{m} (a_1 \leftrightarrow a_3) = \l( \Reff^{m^*} (a_2^* \leftrightarrow a_4^*) \r)^{-1}.\]
\end{lemma}

We fix $p \geq 2$ and $p_1, p_2, p_3, p_4 \geq 1$ such that $\sum_{i=1}^4 p_i = p+2$. We recall that $Q_p$ is a critical Boltzmann quadrangulation with a simple boundary of length $2p$, and that $\Ti{Q_p}$ is the finite map obtained from $Q_p$ by the \Tutte (see Figure \ref{fig_duality_mp}). In particular, the vertices of $\Ti{Q_p}$ are the white vertices of $Q_p$. We represent the boundary of $Q_p$ as a rectangle in such a way that the top-left corner is a white vertex and the top (resp. left, bottom, right) side has $2p_1-1$ edges (resp. $2p_2-1$, $2p_3-1$, $2p_4-1$), as on the left of Figure \ref{fig_duality_mp}. We denote by $A_1$ (resp. $A_3$) the set of white vertices on the top (resp. bottom) part of the boundary, and by $A_2^*$ (resp. $A_4^*$) the set of black vertices on the left (resp. right) part of the boundary. We also write
\[R_{p_1, p_2, p_3, p_4}= \Reff^{\Ti{Q_p}} \l( A_1 \leftrightarrow A_3 \r).\]

\begin{lemma}\label{lem_duality_random}
Assume that $\max(p_1,p_3) \leq \min(p_2,p_4)$. Then we have
\[\P \l( R_{p_1, p_2, p_3, p_4} \geq 1 \r) \geq \frac{1}{2}.\]
\end{lemma}

\begin{proof}
The proof relies on two remarks. First, the resistances of the form $R_{p_1, p_2, p_3, p_4}$ satisfy a monotonicity property: if $\sum_{i=1}^4 p_i = \sum_{i=1}^4 q_i=p+2$ with $p_1 \leq q_1$ and $p_3 \leq q_3$, but $p_2 \geq q_2$ and $p_4 \geq q_4$, then $R_{p_1, p_2, p_3, p_4}$ stochastically dominates $R_{q_1, q_2, q_3, q_4}$. Indeed, it is possible to define subsets $B_1, B_2^*, B_3, B_4^*$ of $\partial Q_p$ of sizes $q_1, \dots, q_4$ in a similar way as $A_1, A_2^*, A_3, A_4^*$, in such a way\footnote{More precisely, if $p_1+p_2 \geq q_1+q_2$, we can place the separation between $B_1$ and $B_4^*$ at the same place as the separation between $A_1$ and $A_4^*$. If this is not the case, then we have $p_3+p_4 \geq q_3+q_4$ and we place the separation between $B_2^*$ and $B_3$ at the same place as the separation between $A_2^*$ and $A_3$.} that $A_1 \subset B_1$ and $A_3 \subset B_3$. Therefore, we have 
\[ \Reff^{\Ti{Q_p}} \l( A_1 \leftrightarrow A_3 \r) \geq \Reff^{\Ti{Q_p}} \l( B_1 \leftrightarrow B_3 \r), \]
so it is possible to couple $R_{p_1, p_2, p_3, p_4}$ and $R_{q_1, q_2, q_3, q_4}$ in such a way that the first one is at least the second.

The second remark is that Lemma \ref{lem_duality_resistances} implies that $R_{p_2, p_3, p_4, p_1}$ has the same distribution as $R^{-1}_{p_1,p_2,p_3,p_4}$. Once this is proved, the conclusion is easy. Indeed, the monotonicity shows that $R_{p_1,p_2,p_3,p_4}$ dominates $R_{p_2, p_3, p_4, p_1}$, so $R_{p_1,p_2,p_3,p_4}$ dominates its inverse, so it has probability at least $\frac{1}{2}$ to be at least $1$.

We finally explain in details how we use Lemma \ref{lem_duality_resistances}. Let $\Tti{Q_p}$ be the map obtained from $\Ti{Q_p}$ by contracting on the one hand all the vertices of $A_1$ into a vertex $a_1$, and on the other hand all the vertices of $A_3$ into a vertex $a_3$. Then $R_{p_1,p_2,p_3,p_4}$ is equal to the resistance between $a_1$ and $a_3$ in $\Tti{Q_p}$, and we are precisely in the context of Lemma \ref{lem_duality_resistances}. More precisely, let $\Tti{Q_p}^*$ be the dual map of $\Tti{Q_p}$, whose vertices correspond to the inner faces of $\Tti{Q_p}$, and to both sides of its outer face (denoted by $a^*_2$ and $a^*_4$). By Lemma \ref{lem_duality_resistances}, we have
\[ \Reff^{\Tti{Q_p}} (a_1 \leftrightarrow a_3)= \Reff^{\Tti{Q_p}^*} (a_2^* \leftrightarrow a_4^*)^{-1}.\]
On the other hand, as can be seen on the right of Figure \ref{fig_duality_mp}, the map $\Tti{Q_p}^*$ can be obtained from $Q_p$ in the exact same way as $\Tti{Q_p}$ by exchanging the roles of the black and white vertices, and by contracting the vertices of $A_2$ into $a^*_2$ and the vertices of $A_4$ into $a^*_4$. Since the distribution of $Q_p$ is invariant under translating the root an exchanging black and white vertices, we deduce that $\Tti{Q_p}^*$ has the same distribution as $\Tti{Q_p}$ obtained for $(p_2,p_3,p_4,p_1)$, so $\Reff^{\Tti{Q_p}^*} (a_2 \leftrightarrow a_4)$ has the same distribution as $R_{p_2,p_3,p_4,p_1}$, which concludes the proof.

\begin{figure}
\begin{center}
\includegraphics[width=\textwidth]{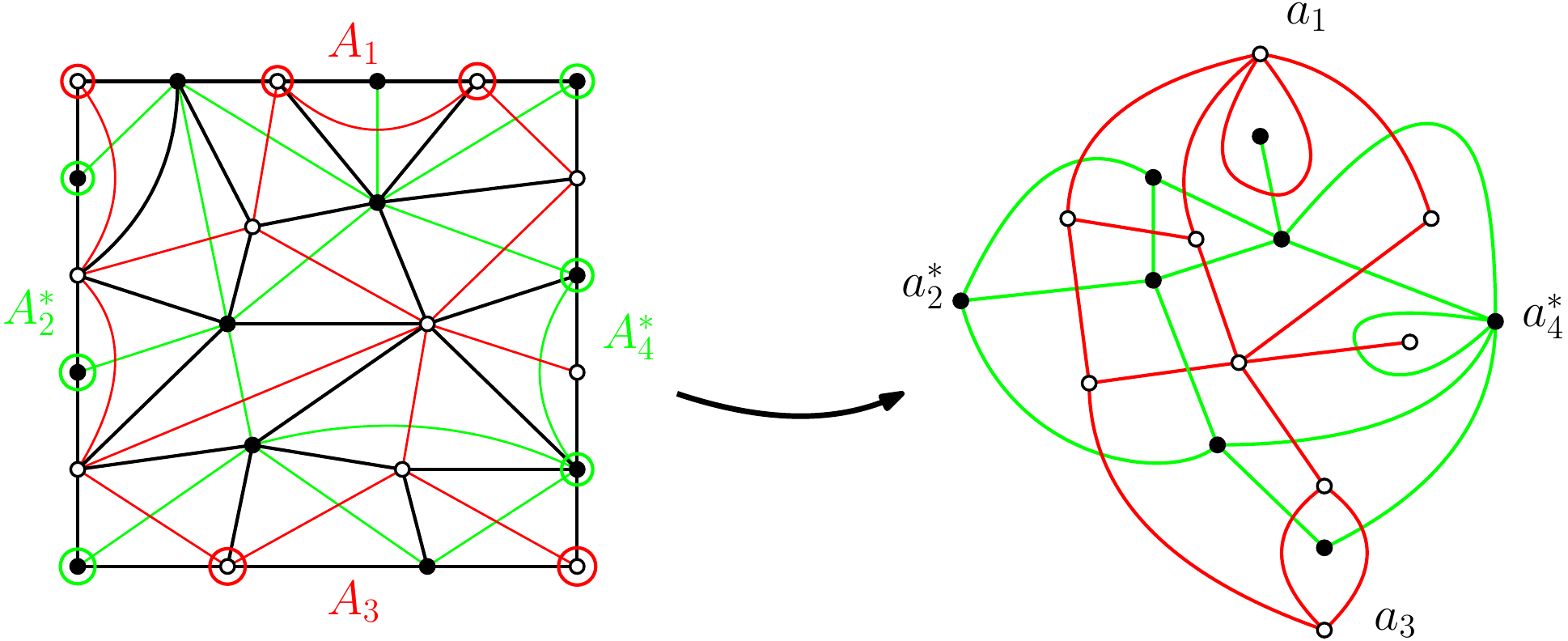}
\end{center}
\caption{On the left, the quadrangulation $Q_p$ (in black), the map $\Ti{Q_p}$ (in red), and its dual $\Ti{Q_p}^*$ (in green). The vertices of $A_1$ and $A_3$ are circled in red, and the vertices of $A_2^*$ and $A_4^*$ are circled in green. Here we have $p_1=p_2=p_4=3$ and $p_3=2$. On the right, the maps $\Tti{Q_p}$ (in red) and its dual $\Tti{Q_p}^*$ (in green). We see that they are built from $Q_p$ in the same way, up to exchanging the roles of white and black vertices.}\label{fig_duality_mp}
\end{figure}
\end{proof}

\begin{rek}
The only property of $Q_p$ that we used here is the invariance of its distribution under root translation and exchange of the colors. In particular, it remains true if $Q_p$ is Boltzmann but non-critical, if it has a fixed number of vertices, or if it is biased by the partition function of some statistical physics model.
\end{rek}

\section{Peeling estimates in the UIHPQ}
\label{Sec_peeling}

As explained in the Introduction, the idea of the proof of Theorem \ref{main_thm} is to find infinitely many disjoint, independent Boltzmann “blocks”, which all separate the root from infinity, and to apply Lemma \ref{lem_duality_random} to each of them. Our goal in this section is to explain how we find one such block. The properties that we want this block to satisfy are summed up in Proposition \ref{Prop block construction}. Our main tool for this is the peeling process of the \UIHPQ $Q_{\infty}$ (see Section \ref{subsec_UIHPQ}).

In all this section, we fix an arbitrary even integer $L \geq 2$, which should be seen as a “scale parameter”. In Section \ref{Sec_proof_main_theorem} we will repeatedly use Proposition \ref{Prop block construction} for exponentially increasing values of $L$. We recall from Figure \ref{fig_submap_quadrangulation} that a sub-map $q$ of $Q_{\infty}$ is a finite part of $Q_{\infty}$ which may have holes filled with finite connected components of $Q_{\infty} \backslash q$. We are particularly interested in the case where the root vertex $\rho$ of $Q_\infty$, as well as the first $L$ edges on its right along $\partial Q_{\infty}$, lie on the boundary of a hole $h$ of $q$. In this case, we split $\partial h$ into four segments:
\begin{itemize}
\item
the top boundary $\partial_t h$, which is the set of edges of $\partial h$ that are not in $\partial Q_{\infty}$,
\item
the left boundary $\partial_{\ell} h$, which is the set of edges of $\partial h \cap \partial Q_{\infty}$ on the left of $\rho$,
\item
the bottom boundary $\partial_b h$, which consists of the first $L$ edges of $\partial Q_\infty$ on the right of $\rho$,
\item
the right boundary $\partial_{r} h$, which is the set of edges of $\partial h \cap \partial Q_{\infty}$ on the right of $\rho$ which do not belong to $\partial_b h$.
\end{itemize}
For example, on Figure \ref{fig_submap_quadrangulation}, if $L=2$, we have $|\partial_t h|=3$, $|\partial_{\ell} h|=1$, $|\partial_{b} h|=L=2$ and $|\partial_r h|=2$.

\begin{proposition}
\label{Prop block construction}
There are constants $C, \delta>0$ independent of $L$ such that the following holds. We can build a finite sub-map $Q$ of $Q_{\infty}$ such that
\begin{enumerate}
\item\label{prop_item_markov}
conditionally on $Q$, the connected components of $\uhpqs \backslash Q$ are independent, the infinite connected component has the same distribution as $\uhpqs$ and the finite components are critical Boltzmann quadrangulations with the right perimeters;
\item\label{prop_item_dimensions_hole}
with probability at least $\delta$, the sub-map $Q$ has a hole $H$ such that the root edge of $Q_\infty$, as well as the $L$ edges of $\partial Q_\infty$ on its right, lie on $\partial H$, and $H$ satisfies $|\partial_t H| \leq L$, $|\partial_{\ell} H| \geq L$, $|\partial_{b} H|=L$ and $|\partial_r H| \geq L$;
\item\label{prop_item_not_too_quick}
we have $\E \left[ \frac{|\partial_t Q|}{L} \right] \leq C$ and $\E \left[ \frac{|\partial_b Q|}{L} \right] \leq C$.
\end{enumerate}
\end{proposition}

Roughly speaking, the role of the second assumption is to allow us to apply Lemma \ref{lem_duality_random} to the quadrangulation filling the hole $H$. The role of the third assumption is to guarantee that $L$ will not grow too quickly when we apply Proposition \ref{Prop block construction} repeatedly in Section \ref{Sec_proof_main_theorem}. If we are only interested in proving the recurrence of $\uhpms$ with no quantitative bound, then the third assumption is not necessary.

Our goal is now to prove Proposition \ref{Prop block construction}. The idea is to explore $\uhpqs$ with the right peeling algorithm, starting at distance $\frac{3}{2}L$ on the left of the root edge, and to stop the exploration when it hits a point of $\partial Q_\infty$ far enough on the right of the root, if this does not occur too late. More precisely, we perform a peeling exploration of the kind described in Section \ref{subsec_UIHPQ}, and denote by $Q^i$ the explored sub-map after $i$ steps. The peeling algorithm that we use is the following: at each step $i \geq 1$, the peeled edge $\AA(Q^{i-1})$ is such that the right endpoint of $\AA(Q^{i-1})$ lies at distance $\frac{3}{2}L$ from the root vertex along the top boundary of $Q^{i-1}$. We would like to stop the exploration when it hits a point of $\partial Q_{\infty}$ which is either on the right of the root, or at distance at most $L$ on its left along $\partial Q_{\infty}$.

However, it is convenient (especially for Item \ref{prop_item_not_too_quick} of Proposition \ref{Prop block construction}) to stop the exploration under wider conditions. To define precisely the time when we stop the exploration, we use the random walk related to the peeling process of $Q_{\infty}$. More precisely, for $i \geq 0$, we recall that $X_i$ is the difference between the lengths of the top and the bottom boundary of $Q^i$. We also write $\Delta X_i=X_i-X_{i-1}$ for every $i\geq 1$. We recall from Section \ref{subsec_UIHPQ} that $X$ is a centered random walk on $\Z$, which means that the increments $\Delta X_i$ are i.i.d. with mean $0$. We also define
\begin{align*}
\ttime &= L^{3/2},\\
\twalk &= \min \{ i \geq 0 \ \big| \ |X_i| \geq \frac{1}{2} L \},\\
\tjump &= \min \{ i \geq 1 \ \big| \  |\Delta X_i| \geq \frac{1}{2} L \},\\
\tau &= \min(\ttime, \twalk, \tjump).
\end{align*}

The sub-map $Q$ in the statement of Proposition \ref{Prop block construction} is the map $Q^{\tau}$, i.e. the explored map at time $\tau$. As explained in the very end of Section \ref{subsec_UIHPQ}, we recall that the finite holes formed at time $\tau$ are not filled, so the map $Q^{\tau}$ may have one or two finite holes (one of them will be the $H$ of Item \ref{prop_item_dimensions_hole} of Proposition \ref{Prop block construction}). Item \ref{prop_item_markov} of Proposition \ref{Prop block construction} is then an immediate consequence of the spatial Markov property of $\uhpqs$.
Note also that for every $i$, the right end of the peeled edge $\AA(Q^i)$ always lies at distance $\frac{3}{2}L$ from the root along the boundary of the infinite hole of $Q^i$. Hence, if a boundary point on the right of the root or at distance $\leq L$ on its left is hit, then at least $\frac{1}{2} L$ boundary edges are swallowed. Therefore, we have $\Delta X_i \leq - \frac{1}{2}L$ and the exploration is indeed stopped. In particular, for $i<\tau$, the explored map $Q^i$ does not contain the root $\rho$.

We now check Item \ref{prop_item_dimensions_hole} of Proposition \ref{Prop block construction}. We recall that $F^i$ is the peeled face at time $i$, i.e. $F^i \in Q^{i} \backslash Q^{i-1}$ for $i \geq 1$. We call the exploration \emph{successful} if $\tau=\tjump$ and if the quadrangle $F^{\tau}$ has a vertex on $\partial \uhpqs$, at distance at least $3L$ on the right of the root. Item \ref{prop_item_dimensions_hole} is the combination of the next two lemmas.

\begin{lemma}\label{lem_successful_hole}
On the event where the exploration is successful, the probability that Item \ref{prop_item_dimensions_hole} of Proposition \ref{Prop block construction} is not satisfied goes to $0$ as $L \to +\infty$.
\end{lemma}

\begin{lemma}\label{lem_probability_successful}
There is a constant $\delta>0$ such that, for every $L \geq 1$,
\[ \P \l( \mbox{the exploration is successful} \r) \geq 2 \delta.\]
\end{lemma}

Lemma \ref{lem_successful_hole} is almost deterministic, whereas the proof of Lemma \ref{lem_probability_successful} relies on the convergence of the walk $X$ to a Lévy process.

\begin{proof}[Proof of Lemma \ref{lem_successful_hole}]
The only non-deterministic part of the proof is to show that with high probability, the large hole created in the last step of a successful exploration is not actually split into two large holes.

Let us work on the event where the exploration is successful. Let $a$, $b$, $c$, $d$ be the four vertices (in counterclockwise order) of the quadrangle $F^{\tau}$, where $a$ and $b$ are incident to the peeled edge $\AA(Q^{\tau-1})$. Let also $Q^{\tau-1}_\infty$ be the infinite quadrangulation filling the infinite hole of $Q^{\tau-1}$, rooted at the same edge as $Q_\infty$ (we know the root of $Q_\infty$ belongs to $\partial Q^{\tau-1}_\infty$ by the discussion right after the definition of $\tau$). The only thing to rule out is the possibility that both $c$ and $d$ lie on $\partial Q^{\tau-1}_\infty$, but one of them is far on the right and the other too close to the root. More precisely, assume that $d$ is at distance at least $3L$ on the right of the root, but $c$ is on $\partial Q^{\tau-1}_\infty$, either on the left of the root at distance at most $L$, or on its right at distance at most $2L$. 
If this occurs, the quadrangle $abcd$ swallows two finite holes of perimeter $\ell_1+1$ and $\ell_2+1$ (one from $b$ to $c$ and one from $c$ to $d$), where $\ell_1$ is the distance from $b$ to $c$ along the boundary, and $\ell_2$ the distance from $c$ to $d$ along the boundary. Moreover, we must have $\ell_1 \geq \frac{L}{2}$ and $\ell_2 \geq L$. Therefore by the estimate \eqref{eqn_bad_peeling_case}, the probability for this bad event to occur at some fixed time is bounded by
\[ \frac{1}{4} \sum_{\ell_1, \ell_2 \geq L/2} \l( \frac{2}{9} \r)^{\ell_1+\ell_2} \monoQQ_{\ell_1+1} \monoQQ_{\ell_2+1} \leq C \sum_{\ell_1, \ell_2 \geq L/2} \ell_1^{-5/2} \ell_2^{-5/2} \leq C' L^{-3}\]
for some absolute constants $C$ and $C'$. Since $\tau \leq L^{3/2}$ by definition, the probability for such a peeling step to occur before $\tau$ is $O(L^{-3/2}) = o(1)$.

Therefore, up to an event of probability $o(1)$ as $L\to\infty$, if the exploration is successful, then one of the vertices $c$ and $d$ is at distance at least $3L$ on the right of the root and the other one is:
\begin{itemize}
\item
either not on $\partial Q^{\tau-1}_\infty$,
\item
or at distance at least $L$ on the left of the root along $\partial Q^{\tau-1}_\infty$,
\item
or at distance at least $2L$ on its right.
\end{itemize}
Hence, the last exploration step creates a hole around the root vertex, that we denote by $H$. 
We now finish the proof by checking that $H$ satisfies Item \ref{prop_item_dimensions_hole} of Proposition \ref{Prop block construction}. 
The first explored vertex of $\partial Q^{\tau-1}_\infty$ on the right of $\rho$ (which is also the first explored vertex of $\partial Q_\infty$ on the right of $\rho$) is at distance at least $2L$ of $\rho$ along $\partial Q_\infty$, so we have $|\partial_r H| \geq 2L-L=L$. Moreover, by the definition of $\tau$, the segment of $\partial \uhpqs$ of length $L$ on the left of $\rho$ has not been touched, so it is part of the boundary of $H$, so $|\partial_{\ell} H| \geq L$. Finally, $\partial_t H$ consists of one or two edges of the quadrangles $abcd$, together with the segment of $\partial_t Q^{\tau-1}$ between the peeled edge $ab$ (excluded) and the rightmost point hit by the exploration before time $\tau-1$. We know by the choice of our peeling algorithm that $b$ is at distance exactly $\frac{3}{2}L$ from the root along $\partial Q^{\tau}_\infty$. On the other hand, by definition of $\tau$, the rightmost point hit by the exploration before time $\tau-1$ is at distance at least $L$ from the root along $\partial Q^{\tau-1}_\infty$. It follows that $|\partial_t H| \leq \frac{1}{2}L+2 \leq L$.
\end{proof}

\begin{proof}[Proof of Lemma \ref{lem_probability_successful}]
Assume that $\tau=\tjump$ and that $\Delta X_\tau \leq -10L$. Then the last peeling step swallows at least $10L$ vertices, on the left or on the right of the peeled edge. By symmetry, with probability $\frac{1}{2}$, at least $5L$ of these vertices are on the right, so we have
\[\P \l( \mbox{the exploration is successful} \r) \geq \frac{1}{2} \P \l( \tau=\tjump, \Delta X_i \leq -10L \r).\]
In particular, the event in the right-hand side only depends on the walk $X$. We would like to compare it with the same event for a Lévy process $S$. As recalled in Section \ref{subsec_UIHPQ}, we know from \cite[Section 6.2]{CLGpeeling} that
\begin{equation}\label{eqn_convergence_walk_levy}
\l( \frac{X_{\floor{L^{3/2}t}}}{L} \r)_{0 \leq t \leq 1} \xrightarrow[L \to +\infty]{} \l( S_t \r)_{0 \leq t \leq 1},
\end{equation}
where $S$ is a spectrally negative stable Lévy process with index $3/2$.

Let $\mathscr{A}$ be the set of càdlàg functions $f$ on $[0,1]$ such that the first jump of $f$ of magnitude larger than $\frac{1}{2}$ is a negative jump of magnitude at least $10$, and $f$ stays in $\l( -\frac{1}{2}, \frac{1}{2} \r)$ before that jump. For $f$ càdlàg and $t \in (0,1]$, we write $\Delta f(t)=f(t)-\lim_{s \to t^-} f(s)$. We have
\[ \partial \mathscr{A} \subset \{ \exists t, \Delta f(t)=-10 \} \cup \{ \exists t, \Delta f(t)=-\frac{1}{2} \} \cup \left\{ \sup_{[0, \tjump)} |f|=\frac{1}{2} \right\} \]
for the Skorokhod topology, so $\P \l( S \in \partial \mathscr{A} \r)=0$. By \eqref{eqn_convergence_walk_levy} and the portmanteau theorem, we obtain
\[ \P \l( \l( \frac{X_{\floor{L^{3/2}t}}}{L} \r)_{0 \leq t \leq 1} \in \mathscr{A} \r) \xrightarrow[L \to +\infty]{} \P \l( L \in \mathscr{A} \r) >0, \]
which is enough to conclude.
\end{proof}

We now move on to Item \ref{prop_item_not_too_quick} of Proposition \ref{Prop block construction}. As for Item \ref{prop_item_dimensions_hole}, the idea is to control the quantities we are interested in (here $|\partial_b Q^{\tau}|$ and $|\partial_t Q^{\tau}|$) in terms of the random walk $X$. The useful point is that, with the peeling algorithm $\AA$ that we chose, the bottom length $|\partial_b Q^i|$ is closely related to the running minimum of the walk $X$.

\begin{lemma}\label{lem_boundary_and_infimum}
We have
\begin{equation*}
|\partial_b Q^{\tau}| \leq \frac{3}{2}L+2-2\min_{[0, \tau]} X.
\end{equation*}
\end{lemma}

Note that, once this is proved, we can easily bound $|\partial_t Q^{\tau}|$. Indeed, we have $X_{\tau}=|\partial_t Q^{\tau}| - |\partial_b Q^{\tau}|$ by definition of $X$, so
\begin{equation*}
|\partial_t Q^{\tau}| \leq X_{\tau}+\frac{3}{2} L+2-2 \min_{[0,\tau]} X \leq 2L+4-2\min_{[0,\tau]} X,
\end{equation*}
since $X_{\tau} \leq X_{\tau-1}+2 \leq \frac{L}{2}+2$. The proof of Lemma \ref{lem_boundary_and_infimum} relies on our choice of the peeling algorithm and of the stopping time $\tau$. It is completely deterministic.

\begin{proof}[Proof of Lemma \ref{lem_boundary_and_infimum}]
We recall that $\partial_b Q^{\tau}$ is the largest segment of $\partial Q_{\infty}$ containing the root vertex $\rho$ and all the vertices incident to a face of $Q^{\tau}$. Let $x_{\ell}$ and $x_r$ be respectively the leftmost and the rightmost vertices of $\partial_b Q^{\tau}$. Note that $x_r$ is either $\rho$, or is hit for the first time at time $\tau$. For $0 \leq i \leq \tau$, we denote by $Q_{\infty}^i$ the infinite hole of $Q^i$. If $x$ and $y$ are two vertices of $\partial Q^i_{\infty}$, it will be convenient to denote by $d_i(x,y)$ the distance between $x$ and $y$ along $\partial Q^i_{\infty}$, i.e. the number of edges of $\partial Q^i_{\infty}$ lying between $x$ and $y$.

We first treat the case where $x_{\ell}$ and $x_r$ are discovered (i.e. hit for the first time by the exploration) at the same time\footnote{It could be shown that this case is very unlikely, but it is simpler to treat it deterministically like the other case.}. This time must be $\tau$. If this occurs, then $x_{\ell}$ and $x_r$ are both incident to the face that is explored at time $\tau$, so $|\partial_t Q^{\tau}| \leq 2$ and
\[ |\partial_b Q^{\tau}|=|\partial_t Q^{\tau}|-X_{\tau} \leq 2-\min_{[0,\tau]} X ,\]
so the conclusion of Lemma \ref{lem_boundary_and_infimum} holds.

We now assume that $x_{\ell}$ and $x_r$ are discovered at different times. We first estimate the position of $x_{\ell}$. Let $0 \leq j \leq \tau$ be the time at which $x_{\ell}$ is hit for the first time. 
We recall that $F^j \in Q^j \backslash Q^{j-1}$ is the quadrangular face that is explored at time $j$. Among all the vertices incident to $F^j$ that belong to $\partial Q^{j-1}_{\infty}$, let $a$ be the rightmost one (see the left part of Figure \ref{fig_proof_bounding_bottom_boundary}). Then $a$ lies on the right of the edge that is peeled at time $j$ (by definition of $a$) and on the left of $\rho$ (if not we would have $j=\tau$, and $x_{\ell}$ and $x_r$ would be hit first at the same time), so $d_j(\rho, a) \leq \frac{3}{2}L$. Therefore, we have
\begin{align*}
d_0(\rho,x_{\ell}) &= d_j(\rho, x_{\ell})-X_j \\
& \leq d_j(\rho,a)+d_j(a,x_{\ell})-X_j \\
& \leq \frac{3}{2}L+2-X_j \\
& \leq \frac{3}{2}L+2-\min_{[0,\tau]} X.
\end{align*}

\begin{figure}
\begin{center}
\includegraphics[width=\textwidth]{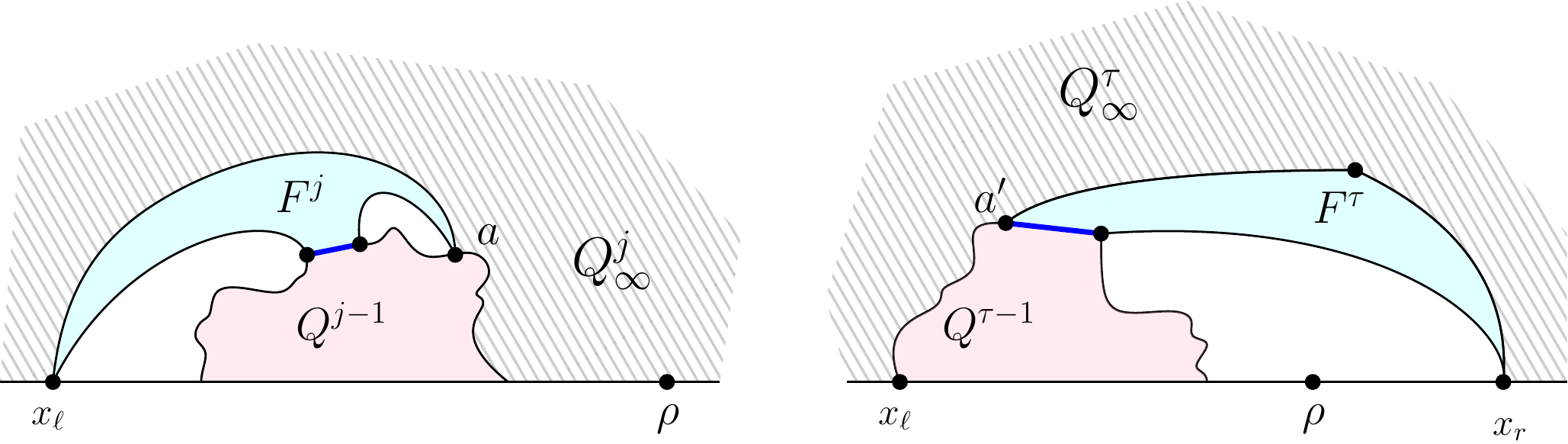}
\end{center}
\caption{Proof of Lemma \ref{lem_boundary_and_infimum}. In this drawing we assume that $x_\ell$ and $x_r$ are first discovered at different times. On the left, the step $j$ at which the leftmost vertex $x_{\ell}$ is discovered. On the right, the last step $\tau$, where $x_r$ is discovered.}\label{fig_proof_bounding_bottom_boundary}
\end{figure}

We now study $x_r$. As noted above, either $x_r=\rho$ or $x_r$ is hit exactly at time $\tau$. In the first case, we have \[|\partial_b Q^{\tau}|=d_0(\rho,x_{\ell}) \leq \frac{3}{2}L+2-\min_{[0,\tau]} X,\]
and the lemma holds. We now focus on the second case. Note that in this case, the vertex $x_{\ell}$ has been hit for the first time strictly before time $\tau$. Recall that $F^{\tau}$ is the quadrangular face that is explored at time $\tau$. Among all the vertices of $F^{\tau}$ that belong to $\partial Q^{\tau-1}_{\infty}$, let $a'$ be the leftmost one (see the right part of Figure \ref{fig_proof_bounding_bottom_boundary}). By definition of $x_{\ell}$, the vertex $a'$ lies on the right of $x_{\ell}$, and the segment of the top boundary between $x_{\ell}$ and $a'$ is not changed between times $\tau-1$ and $\tau$. Hence, we can write
\begin{align*}
|\partial_b Q^{\tau}| &= |\partial_t Q^{\tau}|-X_{\tau}\\
&=d_{\tau}(x_{\ell},a')+d_{\tau}(a',x_r)-X_{\tau}\\
&\leq d_{\tau-1}(x_{\ell},a')+2-X_{\tau}.
\end{align*}
Moreover, by definition, the vertex $a'$ lies on the left of the peeled edge on $\partial_t Q^{\tau-1}$, so we have $d_{\tau-1}(x_{\ell},a') \leq |\partial_t Q^{\tau-1}|-\frac{3}{2} L$. Hence, we can write
\begin{align*}
|\partial_b Q^{\tau}| &\leq |\partial_t Q^{\tau-1}|-\frac{3}{2}L+2-X_{\tau}\\
&= |\partial_b Q^{\tau-1}|+X_{\tau-1}-\frac{3}{2}L+2-X_{\tau}.
\end{align*}
By the definition of the stopping time $\tau$, we have $X_{\tau-1} \leq \frac{1}{2}L$. Moreover, at time $\tau-1$, no vertex on the right of $\rho$ has been hit, so $|\partial_b Q^{\tau-1}|=d_0(\rho,x_{\ell})$. By our estimate on the position of $x_{\ell}$, we finally obtain
\[|\partial_b Q^{\tau}| \leq \l( \frac{3}{2}L+2-\min_{[0,\tau]} X \r) +\frac{1}{2}L-\frac{3}{2}L+2-X_{\tau} \leq \frac{1}{2}L+4-2\min_{[0,\tau]} X,\]
which concludes the proof of the Lemma.
\end{proof}

\begin{rek}
Our computations to prove Lemma \ref{lem_boundary_and_infimum} may seem a bit convoluted, especially in the last part of the proof. The reason why the end of the proof is not obvious is that it is necessary to use the fact that $X_{\tau-1} \leq L/2$. Indeed, if it was not true, we might imagine a case where $|\partial_t Q^{\tau-1}|$ is much larger than $L$, and where $\Delta X_{\tau} \approx - |\partial_t Q^{\tau-1}|$. If this was the case, then $X_{\tau}$ would not be very large compared to $L$, whereas $|\partial_b Q^{\tau}|$ and $|\partial_t Q^{\tau}|$ would both be very large.
\end{rek}

We can now finish the proof of Proposition \ref{Prop block construction} (the only part left is to check Item \ref{prop_item_not_too_quick}). By Lemma \ref{lem_boundary_and_infimum}, it is enough to prove that $\E \left[ \frac{3}{2}-2\frac{\min_{[0, \tau]} X}{L} \right]$ is bounded when $L \to +\infty$. For this, we need to bound $\P \l( \min_{[0, \tau]} X \leq -aL \r)$ uniformly in $L$. Note that by definition of $\tau$, we have $X_i \geq -\frac{1}{2}L$ for every $i<\tau$. Moreover, if $\tau \ne \tjump$, then
\[ X_{\tau}=X_{\tau-1}+\Delta X_{\tau} \geq -\frac{1}{2}L-\frac{1}{2}L=-L,\]
so $\min_{[0,\tau]} X \geq -L$. Therefore, if we choose $a>1$, we have
\begin{align*}
\P \l( \min_{[0, \tau]} X \leq -aL \r) &= \P \l( \tau = \tjump, X_{\tau} \leq -aL \r)\\
& \leq \P \l( \Delta X_{\tau} \leq -(a-1) L \r)\\
& \leq \P \l( \exists i \in [1, L^{3/2}], \Delta X_i \leq -(a-1)L \r)\\
& \leq L^{3/2} \P \l( X_1 \leq -(a-1)L \r)\\
& \leq \frac{c L^{3/2}}{\l( (a-1)L \r)^{3/2}} \\
&= \frac{c}{(a-1)^{3/2}},
\end{align*}
where $c$ is an absolute constant, and we used the fact that $X_1$ lies in the domain of attraction of a $3/2$-stable law (see \cite[Section 6.2]{CLGpeeling}). From here, we obtain
\begin{align*}
\E \left[ -\frac{\min_{[0, \tau]} X}{L} \right] &= \int_0^{+\infty} \P \l( \min_{[0, \tau]} X \leq -a L \r)\,\mathrm{d}a\\
& \leq 1+\int_1^{+\infty} \max \l( 1, \frac{c}{(a-1)^{3/2}} \r) \, \mathrm{d}a.
\end{align*}
Since this integral converges, we are done.

\section{Proof of Theorem \ref{main_thm}}
\label{Sec_proof_main_theorem}

We now combine Lemma \ref{lem_duality_random} with Proposition \ref{Prop block construction} to prove our main theorem. We denote by $\uhpmg$ the image of $Q_{\infty}$ by the \Tutte, and recall that $M_{\infty}$ is obtained from $\uhpmg$ by removing the finite beads along its boundary.

\begin{lemma}\label{almost_main_to_main}
To prove Theorem \ref{main_thm} for $\uhpms$, it is enough to prove it for $\uhpmg$.
\end{lemma}

\begin{proof}
Assume the theorem is proved for $\uhpmg$. We may see $\uhpms$ as a subgraph of $\uhpmg$. Denote the root of $\uhpmg$ by $\rho^g$ and the root of $\uhpms$ by $\rho$. Then $\rho$ is also the last pinchpoint of $\uhpms$ separating $\rho^g$ from infinity. Let $D$ be the graph distance from $\rho^g$ to $\uhpms$. Then for $r$ large enough, the ball of radius $r$ in $\uhpms$ is equal to the ball of radius $r+D$ in $\uhpmg$, minus the part between $\rho^g$ and $\rho$ and possibly parts of finite beads of $M_\infty^g$. This implies that the same is true for the hulls $B_r^{\bullet}$. Hence we have
\begin{align*}
R_{\eff}^{\uhpms} \l( \rho \leftrightarrow \partial B_r^{\bullet}(\uhpms) \r) &= R_{\eff}^{\uhpmg} \l( \rho^g \leftrightarrow \partial B_r^{\bullet}(\uhpms) \r) - R_{\eff}^{\uhpmg} \l( \rho^g \leftrightarrow \rho \r) \\
&\geq R_{\eff}^{\uhpmg} \l( \rho^g \leftrightarrow \partial B_{r+D}^{\bullet}(\uhpmg) \r) - R_{\eff}^{\uhpmg} \l( \rho^g \leftrightarrow \rho \r) \\
&\geq c \log r -c'
\end{align*}
a.s. for every large enough $r$, where $c$ is given by Theorem \ref{main_thm} for $\uhpmg$, and $c'$ is random but does not depend on $r$. This proves Theorem \ref{main_thm} for $\uhpms$.
\end{proof}

\begin{proposition}\label{almost_main_thm}
Theorem \ref{main_thm} holds if $\uhpms$ is replaced by $\uhpmg$.
\end{proposition}

Proposition \ref{almost_main_thm}, together with Lemma \ref{almost_main_to_main}, gives Theorem \ref{main_thm}. 
The rest of this section is devoted to the proof of Proposition \ref{almost_main_thm}. Proposition \ref{Prop block construction} can be seen as the construction (with positive probability) of a block separating a segment of length $L$ on the boundary from infinity. The idea of the proof is to iterate Proposition \ref{Prop block construction}: the top boundary of the $n$-th block is the segment that the $(n+1)$-th block tries to separate from infinity. More precisely, we define by induction a sequence $(Q[n])_{n \geq 1}$ of finite sub-maps of $\uhpqs$, a sequence $(Q_{\infty}[n])_{n \geq 0}$ of infinite quadrangulations of the half-plane, and a sequence $(L[n])_{n \geq 0}$ of integers as follows (see also Figure \ref{fig_onion}):
\begin{itemize}
\item
$L[0]=2$ and $Q_{\infty}[0]=\uhpqs$;
\item
for every $n \geq 1$, the sub-map $Q[n]$ of $Q_{\infty}[n-1]$ is given by Proposition \ref{Prop block construction}, with $L=L[n-1]$;
\item
for every $n \geq 1$, the map $Q_{\infty}[n]$ is the infinite connected component of $Q_{\infty}[n-1] \backslash Q[n]$, rooted at the leftmost vertex $\rho[n]$ of $\partial_t Q[n]$. 
Moreover, we denote by $\Gamma[n]$ the smallest segment of $\partial Q_{\infty}[n]$ that separates all the $Q[i]$ for $1 \leq i \leq n$ from infinity in $Q_\infty$, and $L[n]$ is the length of $\Gamma[n]$, increased to the next even integer.
\end{itemize}
Finally, we also denote by $\FF[n]$ the $\sigma$-algebra generated by $Q_{\infty} \backslash Q_{\infty}[n]$, i.e. the first $n$ blocks together with the fillings of all the finite holes they form.

\begin{figure}
\begin{center}
\includegraphics[width=0.8\textwidth]{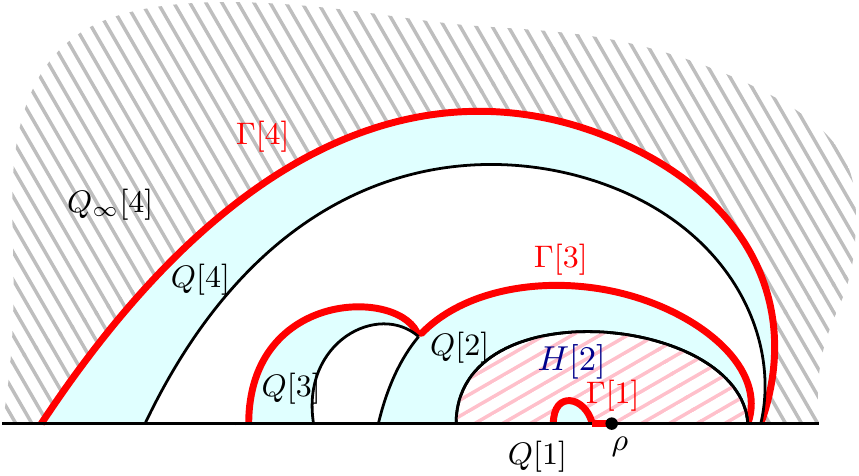}
\end{center}
\caption{The different blocks $Q[n]$ (in light blue). On this example, the only good block is $Q[2]$: the blocks $Q[1]$ and $Q[3]$ do not separate the root from infinity. The block $Q[4]$ does, but its hole has a too small right boundary compared to $L[3]$. The hole $H[2]$ is hatched in pink, and $Q_{\infty}[4]$ in grey.}\label{fig_onion}
\end{figure}

By the first point of Proposition \ref{Prop block construction}, for every $n \geq 1$, the map $Q_{\infty}[n]$ is independent of $\FF[n]$, and has the same distribution as $\uhpqs$. This guarantees that we can apply Proposition \ref{Prop block construction} to $Q_{\infty}[n]$ to obtain $Q[n+1]$. The proof of Proposition \ref{almost_main_thm} can be split into Lemmas \ref{lem_reff_root_an} and \ref{lem_distance_root_an} below. Note that if we are only interested in the recurrence of $\uhpmg$ or $\uhpms$ and not on quantitative resistance bounds, then Lemma \ref{lem_reff_root_an} only is sufficient.

\begin{lemma}\label{lem_reff_root_an}
There is a constant $c>0$ such that almost surely, for every $n$ large enough, we have
\[ \Reff^{\uhpmg} \l( \rho \leftrightarrow \Gamma[n] \r) \geq c n. \]
\end{lemma}

\begin{lemma}\label{lem_distance_root_an}
There is a constant $A>1$ such that almost surely, for $n$ large enough, for every white vertex $v \in \Gamma[n]$, we have $d_{\uhpmg}(\rho,v) \leq A^n$.
\end{lemma}

Proposition \ref{almost_main_thm} is an easy consequence of these two results: for each $n$, the set of white vertices of $\Gamma[n]$ separates the root from infinity in $\uhpmg$. By Lemma \ref{lem_distance_root_an}, there is a constant $\alpha>0$ such that, for every $r$ large enough, $\Gamma[\alpha \log r]$ stays in the ball of radius $r$ in $\uhpmg$, so it separates the root from $\partial B_r^{\bullet} (\uhpmg)$. Therefore, by using also Lemma \ref{lem_reff_root_an}, for every $r$ large enough we have
\[\Reff^{\uhpmg} \l( \rho \leftrightarrow \partial B_r^{\bullet}(\uhpmg) \r) \geq \Reff^{\uhpmg} \l( \rho \leftrightarrow \Gamma[\alpha \log r]  \r) \geq c \alpha \log r.\]

\begin{proof}[Proof of Lemma \ref{lem_reff_root_an}]
We call a index $n \geq 1$ \emph{good} if the event in the second point of Proposition \ref{Prop block construction} occurs. If $n$ is a good index, the sub-map $Q[n]$ of $Q_{\infty}[n-1]$ has a finite hole $H[n]$, which is filled with a critical Boltzmann quadrangulation with a simple boundary of length $2p$ for some $p$. We recall (from just before Proposition \ref{Prop block construction}) that the boundary of $H[n]$ can be split into $4$ parts:
\begin{itemize}
\item
the segment of length $L[n]$ on the right of the root of $Q_{\infty}[n-1]$, denoted by $\partial_b H[n]$,
\item
a segment of $\partial Q_{\infty}[n-1]$ on the right of $\partial_{b} H[n]$, denoted by $\partial_r H[n]$,
\item
the segment $\partial H[n] \backslash \partial Q_{\infty}[n-1]$, denoted by $\partial_t H[n]$,
\item
a segment of $\partial Q_{\infty}[n-1]$ on the left of its root, denoted by $\partial_{\ell} H[n]$.
\end{itemize}
Then Item \ref{prop_item_dimensions_hole} of Proposition \ref{Prop block construction} guarantees that we have $\left| \partial_b H[n] \right|=L[n]$ and $\left| \partial_t H[n] \right| \leq L[n]$, but $\left| \partial_{\ell} H[n] \right|, \left| \partial_r H[n] \right| \geq L[n]$, so we can apply Lemma \ref{lem_duality_random}. Let $M[n]$ be the map obtained by applying the \Tutte to the quadrangulation filling the hole $H[n]$. By Lemma \ref{lem_duality_random}, conditionally on $\FF[n-1]$ and on $n$ being a good step, with probability at least $\frac{1}{2}$, we have
\[\Reff^{M[n]} \l( \partial_b H[n] \leftrightarrow \partial_t H[n] \r) \geq 1.\]
We call $n$ \emph{very good} if this is the case. Note that by definition of $L[n-1]$ and our rooting convention for $Q_{\infty}[n-1]$, the segment $\partial_b H[n]$ contains the boundary between $Q_{\infty}[n-1]$ and $Q_{\infty} \backslash Q_{\infty}[n-1]$. On the other hand, the top boundary $\partial_t H[n]$ separates $\partial_b H[n]$ from the top boundary of $Q[n]$, and therefore from $Q_{\infty}[n]$. Therefore, as can be seen on Figure \ref{fig_onion}, for every good $n$, $M[n]$ separates on the one hand the root $\rho$ and the maps $M[i]$ for $i<n$, and on the other hand the $M[i]$ for $i>n$ and infinity. It follows that
\[\Reff^{\uhpmg} \l( \rho \leftrightarrow \Gamma[n] \r) \geq \sum_{i=1}^{n-1} \mathbbm{1}_{\mbox{$i$ is good}} \Reff^{M[i]} \l( \partial_b H[i] \leftrightarrow \partial_t H[i] \r),\]
where the term $i$ is at least $1$ if $n$ is very good. Therefore, if we denote by $N[n]$ the number of very good indices between $1$ and $n-1$, we have
\begin{equation}\label{eqn_pile_of_blocks}
\Reff^{\uhpmg} \l( \rho \leftrightarrow Q_{\infty}[n] \r) \geq N[n].
\end{equation}

Finally, by the second point of Proposition \ref{Prop block construction}, the conditional probability for any index $n$ to be good conditionally on $\FF[n-1]$ is bounded from below by some $\delta>0$. Moreover, since the Boltzmann quadragulations filling the holes for different good steps are independent, the events $\{ \mbox{$n$ is very good} \}$ for $n \geq 1$ dominate i.i.d. events of probability $\frac{\delta}{2}$. Therefore, by the law of large number, \as for every $n$ large enough we have
\[N[n] \geq \frac{\delta}{3} n,\]
which, combined with \eqref{eqn_pile_of_blocks}, proves the lemma.
\end{proof}

We now prove Lemma \ref{lem_distance_root_an}, that is, we control the graph distance between $\rho$ and $Q_{\infty}[n]$ in $M^g_\infty$. The next lemma allows us to control how distances may increase when we apply the Tutte mapping. Its proof is delayed until the end of the section. We label the vertices of $\partial Q_{\infty}$ as $(u_i)_{i \in \Z}$ as on Figure \ref{figure_proof_degree}, so that $u_0$ is the root vertex and the $u_i$ with $i$ even are the white vertices.

\begin{lemma}\label{lem_degree_boundary}
The distance $d_{\uhpmg}(u_0, u_2)$ has exponential tail.
\end{lemma}

Note also that by the invariance of $\uhpqs$ by root translation along the boundary, this remains true if we replace $u_0$ and $u_2$ by any two consecutive white vertices of $\partial \uhpqs$.

\begin{proof}[Proof of Lemma \ref{lem_distance_root_an} using Lemma \ref{lem_degree_boundary}]
The idea of the proof is to bound distances between vertices on the boundary of $\partial Q[n]$ by using paths following the boundary. Although this might seem very crude, we only care about the logarithms of the distances, so bounding $r$ by $r^2$ is not a problem here. A slight difficulty is that we are interested in distances in $\uhpmg$, so we need to make sure that moving around a black vertex of the quadrangulation does not cost too much. This is where we need to use Lemma \ref{lem_degree_boundary}.

We first argue that $L[n]$ grows at most exponentially in $n$. For every $n \geq 0$, we recall that $\FF[n]$ is the $\sigma$-algebra generated by $Q_{\infty} \backslash Q_{\infty}[n]$. By definition of $\Gamma[n]$ and $L[n]$, the segment $\Gamma[n+1]$ is included in the union of $\Gamma[n]$ and the top boundary of $Q[n+1]$ (see Figure \ref{fig_onion}), so $L[n+1] \leq 1+L[n]+|\partial_t Q[n+1]|$. Hence, it follows from Item \ref{prop_item_not_too_quick} of Proposition \ref{Prop block construction} that $\E \left[ \frac{L[n+1]}{L[n]} | \FF[n] \right] \leq C+2$ for every $n \geq 1$, where $C$ is given by Proposition \ref{Prop block construction}. Therefore, we have $\E [L[n]] \leq (C+2)^n$, so $\P \l( L[n] \geq (C+3)^n \r)$ decreases exponentially in $n$, and $L[n] \leq (C+3)^n$ for all large enough $n$. Similarly, we easily obtain $|\partial_b Q[n]| \leq (C+3)^n$ for $n$ large enough.

On the other hand, Lemma \ref{lem_degree_boundary} states that the distance in $\uhpmg$ between any two consecutive white vertices of $\partial Q_{\infty}[n]$ has exponential tail (recall that $Q_{\infty}[n]$ has the same law as $Q_{\infty}$). By a crude union bound, it follows that with probability at least $1-(C+3)^{n+1} e^{-cn^2}$ (for some absolute constant $c>0$), for any two consecutive white vertices of $\partial Q_{\infty}[n]$ at distance at most $(C+3)^{n+1}$ from the root along $\partial Q_{\infty}[n]$, the distance between them in $\uhpmg$ is at most $n^2$. If this is the case and $|\partial_b Q[n+1]| \leq (C+3)^{n+1}$, then the distance in $\uhpmg$ between any two vertices of $\partial_b Q[n+1]$ (and in particular of $\vG[n]$) is at most $n^2 (C+3)^{n+1}$. In particular, this is true for $\rho[n]$ and $\rho[n+1]$, since by definition $\rho[n+1]$ is the leftmost vertex of $\partial_b Q[n+1]$.
Therefore, almost surely, for $n$ large enough, we have
\begin{equation*}
d_{\uhpmg}(\rho[n], \rho[n+1]) \leq n^2 (C+3)^{n+1}.
\end{equation*}
By induction, we obtain that almost surely, for $n$ large enough, for $n$ large enough, we have
\[d_{\uhpmg}(\rho, \rho[n]) \leq (C+4)^n .\]
Since we also have $d_{\uhpmg}(\rho[n],v) \leq n^2 (C+3)^{n+1}$ for $v \in \Gamma[n]$, we conclude
\[d_{\uhpmg}(\rho,v) \leq (C+5)^n\]
for every $v \in \Gamma[n]$, which proves the lemma.
\end{proof}

\begin{proof}[Proof of Lemma \ref{lem_degree_boundary}]
The idea is to discover the neighbourhood of $u_1$ in $\uhpqs$ by a peeling exploration, until there is a set of faces of $Q_\infty$ containing edges of $M^g_\infty$ that allow to go from $u_0$ to $u_2$ in $M^g_\infty$. This argument essentially goes back to \cite{AS03}. More precisely, we use a non-filled-in peeling algorithm, i.e. we do not discover the finite holes cut out by the explored face. The algorithm is the following. Let $e_1$ be the edge of $\partial \uhpqs$ between $u_1$ and $u_2$. At each step $i$, we consider the connected component of the undiscovered part that contains $e_1$ (this component may be finite or infinite). The edge we peel is the edge on the left of $e_1$ along the boundary of this component. We stop the exploration when we discover a face that is incident to $u_2$ (see Figure \ref{figure_proof_degree}).

The probability to end the exploration at some step conditionally on the previous ones depends on the perimeter $2p$ of the hole that we explore ($p$ may be finite or infinite). We call it $q_p$. It is classical that $q_p>0$ for every $p$, and that $q_p \to q_{\infty}>0$ as $p \to \infty$, so this probability is bounded from below by a positive constant. Hence, the number of steps of the exploration is a.s. finite and has exponential tail.

We finally argue that $d_{\uhpmg}(u_0, u_2)$ is bounded by the number of steps of the exploration described above. At each step $i$, let $f_i$ be the unique face that we discover. We color in red the diagonal joining the two white vertices incident to $f_i$ (this diagonal may be a loop). Note that the red diagonals are edges of $\uhpmg$. It is quite easy to see that the $(i+1)$-th red diagonal starts where the $i$-th ends, so the set of red diagonals forms a path (see Figure \ref{figure_proof_degree}). Moreover, the first one we draw is incident to $u_0$ (the first peeled edge is the one between $u_0$ and $u_1$), and the last one is incident to $u_2$. Therefore, there is a path of red diagonals joining $u_0$ to $u_2$, and its length is bounded by the duration of the exploration, so it has exponential tail, which proves the lemma.

\begin{figure}
\begin{center}
\includegraphics[width=0.7\textwidth]{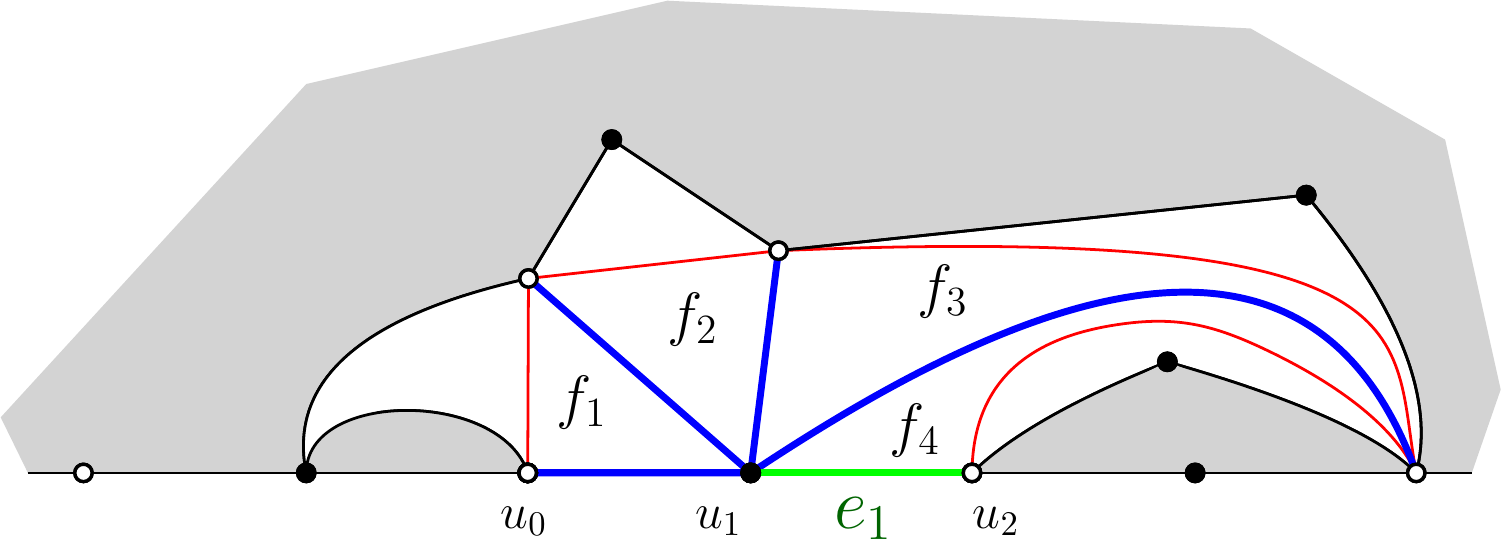}
\end{center}
\caption{Proof of Lemma \ref{lem_degree_boundary}: the red diagonals form a connected sets, and the successive peeled edges are in blue.}\label{figure_proof_degree}
\end{figure}
\end{proof}

\section{The UIHPM as a local limit}
\label{Sec_UIHPM_is_limit_of_Boltzmann_maps}

The goal of this Section is to prove Theorem \ref{thm_defn_UIHPM}. We recall that $M_{\infty}$ is the \UIHPM defined in Section \ref{subsec_Tutte_UIHPM} as the unique infinite bead of $\Ti{Q_{\infty}}$ and that $M_p$ is the critical Boltzmann random map with simple boundary of length $p$. Our goal is to prove the local convergence $M_p \to M_{\infty}$ as $p\to\infty$. The basic strategy of the proof is to use the Tutte mapping $\TT$ of Section \ref{subsec_Tutte_UIHPM} to decompose a critical Boltzmann quadrangulation $Q_p$ into several components. In particular, one of these components will be mapped by $\TT$ to a map with simple boundary that we call the \emph{Tutte core} of $Q_p$, and that will correspond to the infinite bead when $p \to +\infty$. We then prove that the Tutte core has the law of a critical Boltzmann map with a simple boundary (of random length). Finally, we let the boundary length $p$ go to infinity, and we find that the Tutte core converges to the \UIHPM.

While this strategy seems simple, a few technical problems make our proof longer than what might be expected. Two of these issues are that the Tutte core of $Q_p$ has a random perimeter, and that the core is not obvious to define in a canonical way for finite maps (unlike in $\Ti{Q_\infty}$, where it is the only infinite bead). To solve these two problems, we consider random maps $M^{\bullet \bullet, z}$ with randomized perimeters. More precisely, $M^{\bullet \bullet, z}$ is a random map with simple boundary, critical Boltzmann weight on the number of edges, Boltzmann weight $z$ on the perimeter, and three distinct marked edges on the boundary.
 
 This triple marking has two advantages: first, it makes the partition function blow up as $z$ goes to its critical value $z_c=2/9$, which guarantees that the perimeter of $M^{\bullet \bullet, z}$ goes to $+\infty$. Second, marking three boundary edges of a finite map $m$ with non-simple boundary allows us to define the “core with simple boundary” of $m$ as the bead containing the “center of gravity” between the three marked edges. This core with simple boundary plays in $M^{\bullet \bullet, z}$ the role of the infinite bead in $\Ti{Q_\infty}$. This is how we will define the Tutte core of a quadrangulation in Section \ref{Sec_core_decomp_of_quads}. 

We will first prove Proposition \ref{prop_convergence_randomized_perimeters}, which states that
\[M^{\bullet \bullet, z} \xrightarrow[z \to z_c]{(d)} M_{\infty}\]
in the local topology. 
To go from random perimeters to fixed perimeters, we will then use asymptotics of the partition function of $M_p$, which are given by \eqref{asymptotic coefficients_M}.

\subsection{The core decomposition of quadrangulations}
\label{Sec_core_decomp_of_quads}

Let $\triQ$ be the set of finite quadrangulations with a simple boundary and three (pairwise distinct) marked boundary edges $e_1, e_2, e_3$, each oriented from a white to a black vertex so that the external face is on their right, and where $e_1$ is the root edge. Let $q \in \triQ$. For each $i \in \{1,2,3\}$, let $e'_i$ be the edge of $\Ti{q}$ immediately to the left of $e_i$ (when turning around its starting point counterclockwise), oriented in such a way that $e_i$ and $e'_i$ have the same starting point, see \figref{Fig bead and core}.

The beads of $\Ti{q}$ (i.e. the components separated by pinchpoints) have a tree structure, so there is a unique bead such that, if we remove it, none of the connected components contains more than one marked edge. This particular bead is called the \emph{Tutte core} of $q$, and denoted by $\Core(q)$ (see \figref{Fig bead and core}). We will also denote by $\Coreq(q)$ the part of $q$ corresponding to $\Core(q)$ via the Tutte mapping. More precisely, $\Coreq(q)$ is the quadrangulation with boundary obtained by keeping only the faces of $q$ containing an edge of $\Core(q)$, where we only glue faces of $\Coreq(q)$ along edges of $q$ that are inside $\Core(q)$ (the other edges are “cut open”, see Figure \ref{Fig illustration découpe}). We highlight that $\Core(q)$ is a map with a simple boundary, whereas $\Coreq(q)$ is a quadrangulation with truncated boundary, i.e. with simple boundary where all black boundary vertices have degree $2$. We have already noted in Section \ref{subsec_Tutte_UIHPM} that the Tutte mapping is a bijection between maps with simple boundary and quadrangulations with truncated boundary.

The Tutte core $\Core(q)$ is equipped with three marked boundary vertices $v_1, v_2, v_3$ giving the positions of the edges $e'_1, e'_2, e'_3$ (these are the squares on \figref{Fig bead and core}): the vertex $v_i$ is the vertex of $\partial \Core(q)$ which is the closest from the starting point of $e'_i$. It is not always true that the three marked vertices $v_1, v_2, v_3$ are distinct. We denote the set of quadrangulations $q \in \triQ$ where this is the case by $\A$. Our decomposition is simpler to describe if restricted to $\A$, and we will check later that a random quadrangulation with a large boundary length is in $\A$ with high probability. We assume $q \in \A$ from now on.

\paragraph{Recovering quadrangulations from their Tutte core.}
Of course, it is not possible to recover $q$ given only $\Coreq(q)$ (or equivalently given $\Core(q)$), but $\Coreq(q)$ is a part of $q$. Fortunately, there exists a unique way to obtain $q$ by gluing quadrangulations with a simple boundary to edges of the boundary of $\Coreq(q)$. Each of these quadrangulations is glued either along one edge, or along two consecutive boundary edges with the same white endpoint. More precisely, let $v$ be a vertex of $\partial \Core(q)$ (i.e. a white vertex of $\partial \Coreq(q)$). We denote by $e$ (resp. $e'$) the edge of $\partial \Coreq(q)$ just before (resp. just after) $v$ in the trigonometric order. Then we are in one of the two following situations:
\begin{itemize}
\item
one quadrangulation with a simple boundary is glued to the two edges $e$ and $e'$. Its root is glued to $e$;
\item
two quadrangulations with a simple boundary are glued respectively to $e$ and $e'$. The root of the first one is glued at $e$, and we glue to $e'$ the edge of the boundary of the second one that precedes its root (in the trigonometric order).
\end{itemize}
In other words, to each vertex $v$ of $\partial \Core(q)$, we associate an element $q_v$ of $\monoP :=\monoQ \cup \monoQ^2$.
Note that our gluing convention is designed to be consistent with the fact that the root edge is always directed from a white vertex to a black one, and the external face is always on the right of the root edge. We also recall that $\monoQ$ contains the edge-quadrangulation $\dagger$ with boundary length 2 and no inner face.

\begin{figure}[h!]
\begin{center}
\includegraphics[width=0.4\textwidth]{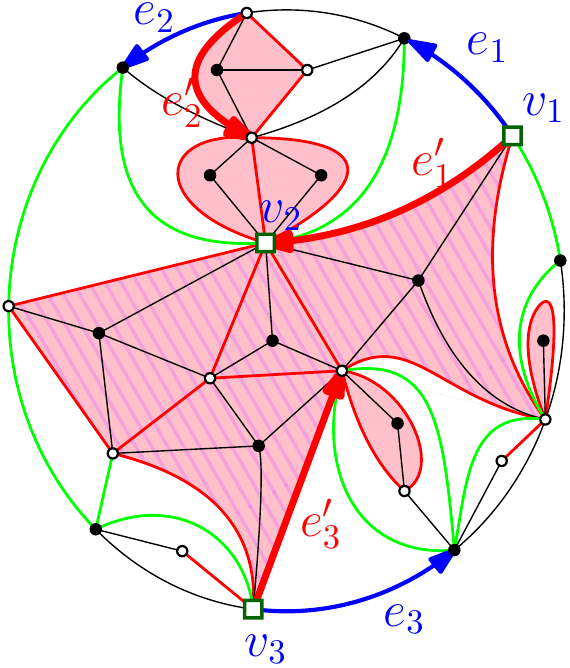}\qquad
\includegraphics[width=0.55\textwidth]{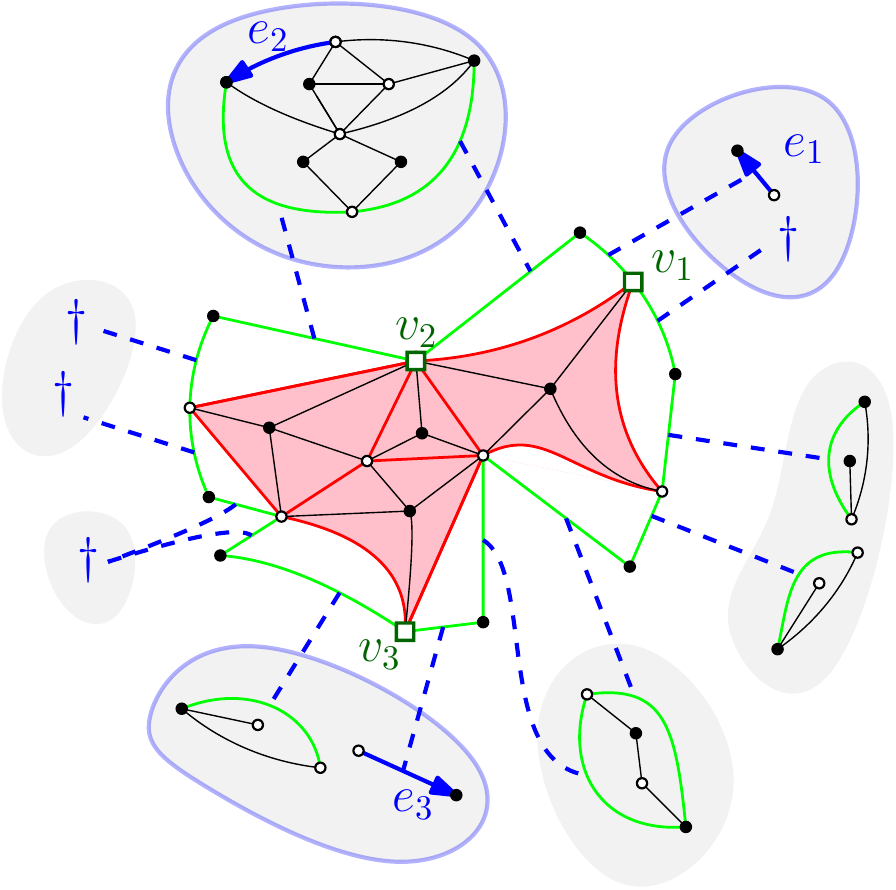}
\end{center}
\caption{The core decomposition of a quadrangulation $q$. On the left, the marked edges of $q$ are in blue, the map $\Ti{q}$ is in red, and the Tutte core is hatched in purple. On the right, we have cut $q$ along the green edges to isolate $\Coreq(q)$ in the center (inside the large green cycle). The root vertices of $\Core(q)$ are the squares. Around the core are the elements of $\monoP$ and $\biP$ that must be glued to $\Coreq(q)$ to recover $q$ (those bearing one of the marked edges, and thus belonging to $\biP$, are circled in blue).}
\label{Fig illustration découpe}
\label{Fig bead and core}
\end{figure}

Finally, the parts corresponding to the marked vertices $v_1, v_2, v_3$ play a special role, because they also need to bear the marked edges $e_1, e_2, e_3$ on their boundaries. If only one quadrangulation $q_i$ is glued to $v_i$, it must have an additional marked edge, which may not be the root of $q_i$ since the root of $q_i$ is glued to $\Coreq(q)$. Hence $q_i$ is an element of the set $\biQ$ of quadrangulations with two distinct marked edges on the boundary, the first one being the root edge. If there are two different quadrangulations $q_i$ and $q'_i$ glued near $v_i$, then (by our assumption that the $v_j$ are distinct) exactly one of them bears an additional marked edge. If it is $q_i$ (the first one in trigonometric order), the marked edge cannot be the root of $q_i$ since the root of $q_i$ is glued to an edge of $\Coreq(q)$, so $q_i \in \biQ$ and $q'_i \in \monoQ$. On the other hand, if the additional marked edge is on $q'_i$, it may or may not be the root of $q'_i$, so $q_i \in \monoQ$ and $q'_i \in \monoQ \cup \biQ$. Therefore, to each of the three vertices $v_1$, $v_2$, $v_3$, we need to associate an element of
\begin{equation}\label{eqn_defn_J}
\biP= \biQ \cup \l( \monoQ \times \biQ \r) \cup \l( \biQ \times \monoQ \r) \cup \monoQ^2.
\end{equation}
See Figure \ref{Fig bead and core} for an example.

For every $p$, let $\triM_p$ be the set of tri-marked finite maps with a simple boundary of length $p$. We have just built an application
\begin{equation}\label{bijection_core_decomposition}
\Psi : \A \longrightarrow \bigcup_{p \geq 3} \triM_p \times \biP^3 \times \monoP^{p-3},
\end{equation}
where we recall that $\A$ is a subset of $\triQ$, $\monoP=\monoQ \cup \biQ$, and $\biP$ is given by \eqref{eqn_defn_J}. The proof that $\Psi$ is a bijection is straightforward, since we have already explained how to glue back the small components to the core.

\subsection{The core decomposition of Boltzmann models}
\label{Sec_Generating_Series}

We now see how the bijection $\Psi$ translates in terms of the generating functions of quadrangulations and maps with a simple boundary. For this, we need to understand the effect of this decomposition on the perimeters and volumes of the objects.

In order to define bivariate generating functions, we use the following combinatorial parameters. If $q$ is a quadrangulation with a boundary, we let its volume parameter $\bv(q)$ be its number of inner faces, and if $m$ is a planar map we let $\bv(m)$ be its total number of edges. Recall that the perimeter $|\partial m|$ of a map $m$ is the degree of its external face. If $q$ is a quadrangulation, we let the boundary length parameter be $\bp(q) \eqdef |\partial q|/2$, and if $m$ is a planar map we let $\bp(m) \eqdef |\partial m|$. 

If $\X$ is a class of maps or quadrangulations, we define its generating function $\XX$ as
\[ \XX(y,z)=\sum_{m \in \X} y^{\bv(m)} z^{\bp(m)}.\]
Note that $\bp$ and $\bv$ are chosen so that this definition of the generating functions matches with the definitions of Section \ref{subsec_basics}.
With the notations of the previous subsection $\X$ may stand for $\monoQ$, $\biQ$, $\triQ$, $\triQA$, $\monoP$,  $\biP$, $\monoM$, $\biM$ or $\triM$, so we have just defined the generating functions $\monoQQ$, $\biQQ$, $\triQQ$, $\triQQA$, $\monoPP$,  $\biPP$, $\monoMM$, $\biMM$ and $\triMM$ (for $\monoPP$ and $\biPP$, we define the perimeter and the volume parameters of a pair of quadrangulations as the sums of the perimeter and volume parameters of its two components).

The generating functions $\monoQQ$ and $\monoMM$ are computed in Section \ref{subsec_basics}.
Note that adding marked edges on the boundary is equivalent to derivating with respect to $z$. More precisely, if $\XX$ is one of the letters $\QQ$ and $\MM$, we have (recall the marked edges have to be distinct):
\[ \XX^{\bullet}(y,z)=z^2 \frac{\partial}{\partial z} \frac{\XX(y,z)}{z} \quad \mbox{ and } \quad \XX^{\bullet \bullet}(y,z)=z^3 \frac{\partial^2}{\partial z^2} \frac{\XX(y,z)}{z}.\]
By the definition of $\monoP$ and $\biP$ \eqref{eqn_defn_J}, we also have
\begin{equation}\label{formula_gen_function_J}
\monoPP=\monoQQ+\monoQQ^2 \mbox{ and } \biPP=2 \monoQQ \biQQ + \monoQQ^2 + \biQQ.
\end{equation}
The regime we are interested in is the regime where $y=\frac{1}{12}$ is fixed and $z$ goes to its critical value. More precisely, the value $y=\frac{1}{12}$ is critical in the sense that $\monoMM(y,z)$ and $\monoQQ(y,z)$ are infinite as soon as $y>\frac{1}{12}$. Moreover, by \eqref{formula_quad_Zp}, we have $\monoQQ \l( \frac{1}{12}, z \r)<+\infty$ if and only if $z\leq \frac{2}{9}$, with
\begin{equation}\label{computation_Q_critical}
\monoQQ\l(\frac{1}{12}, \frac{2}{9}\r) = \frac{1}{3} , \qquad \biQQ\l(\frac{1}{12}, \frac{2}{9}\r) =  \frac{1}{9} \quad \mbox{and} \quad \triQQ\l(\frac{1}{12}, z\r) \underset{z \to 2/9}{\sim} \frac{2 \sqrt{2}}{27 \sqrt{3}} \frac{1}{\sqrt{2/9-z}}. 
\end{equation}
Similarly, using \eqref{generating_function_M}, we have $\monoMM \l(\frac{1}{12},z \r)<+\infty$ if and only if $z \leq 2$, with 
\begin{equation}\label{computation_M_critical}
\monoMM\l(\frac{1}{12}, 2\r) = \frac{1}{3} \quad \mbox{ and } \quad 
 \triMM\l(\frac{1}{12}, z\r) \underset{z \to 2}{\sim} \frac{1}{2} \frac{1}{\sqrt{2-z}}.
\end{equation}

Let us now translate the bijection $\Psi$ of \eqref{bijection_core_decomposition} in terms of generating functions. For $q\in \A$, if $\Psi(q) = \l( m, (c_i)_{1 \leq i \leq 3}, (c'_i)_{1 \leq i \leq \bp(m)-3} \r)$, then we have
\begin{align*}
\bv(q) &= \bv(m) + \sum_{i=1}^3 \bv(c_i) + \sum_{i=1}^{\bp(m)-3} \bv(c'_i),\\
\bp(q) &= \bp(m) + \sum_{i=1}^3 (\bp(c_i)-1) + \sum_{i=1}^{\bp(m)-3} (\bp(c'_i)-1).
\end{align*}
Therefore, by standard algebraic manipulations (see e.g. \cite{FS09}), the bijection $\Psi$ of \eqref{bijection_core_decomposition} translates into the identity
\begin{equation}
\label{Eq_serie_generatrice_A}
\triQQA(y,z) = \triMM\l( y, \frac{1}{z} \monoPP(y,z) \r) \pfrac{\biPP}{\monoPP}^3(y,z).
\end{equation}

This identity on generating functions can be translated into a probabilistic statement about Boltzmann-distributed models. In all that follows, we stay in the case $y=\frac{1}{12}$. If $\X$ is a class of maps or quadrangulations, we define the \emph{$z$-Boltzmann map on $\X$} as the random map $X^{z}$ such that, for every $x \in \X$, we have
\[ \P \l( X^{z}=x \r)= \frac{1}{\XX\l( \frac{1}{12}, z \r)} \l( \frac{1}{12} \r)^{\bv(x)} z^{\bp(x)}.\]
We also denote by $X_p$ the Boltzmann map $X^{z}$ conditionned on $\bp(X^z) = p$. The next result is just the probabilistic translation of \eqref{Eq_serie_generatrice_A}.

\begin{proposition}
\label{Prop_Decomposition_of_Boltzmann_quad_is_Boltzmann}
Let $z \in [0,2/9)$. Then the Tutte core of $A^{z}$ is the $\l( \frac{1}{z} \monoPP \l(\frac{1}{12},z \r) \r)$-Boltzmann map on $\triM$. Moreover, conditionally on the Tutte core having perimeter $p$, if $\Psi(A^z) = \l( M, (C_i)_{1\leq i \leq 3}, (C'_i)_{1 \leq i \leq p-3} \r)$, then
\begin{itemize}
\item $\l( M, (C_i)_{1\leq i \leq 3}, (C'_i)_{1 \leq i \leq p-3} \r)$ is an independent family;
\item the Tutte core $M$ has the law of $M^{\bullet \bullet}_p$;
\item the $C_i$ are \iid $z$-Boltzmann on $\biP$;
\item the $C'_i$ are \iid $z$-Boltzmann on $\monoP$.
\end{itemize}
\end{proposition}

Finally, let us look at what happens when $z \to \frac{2}{9}$: the generating function $\triQQA \l( \frac{1}{12}, z \r)$ goes to $+\infty$ whereas for every fixed $p \geq 1$, the contribution to $\triQQA \l( \frac{1}{12}, z \r)$ of quadrangulations with perimeter $2p$ stays bounded. Hence, we have $\left| \partial A^{z} \right| \to +\infty$ in probability as $z \to \frac{2}{9}$.
Moreover, \eqref{formula_gen_function_J} gives $\frac{1}{2/9} \monoPP \l( \frac{1}{12},\frac{2}{9} \r)=2$ and $\biPP \l( \frac{1}{12},\frac{2}{9} \r)<+\infty$. Hence, when $z \to \frac{2}{9}$, the Tutte core also becomes critical and, for the same reason as $A^{z}$, its perimeter goes to $+\infty$ in probability. On the other hand, the parts glued to the core converge in distribution to the finite maps $J^{2/9}$ and $J^{\bullet, 2/9}$.

\subsection{Convergence of Boltzmann maps with random perimeter}
\label{Sec_cv_Boltzmann_maps}

Our goal is now to prove a version of Theorem \ref{thm_defn_UIHPM} for Boltzmann maps with randomized perimeter. We recall that $M_{\infty}$ is the \UIHPM of Section \ref{subsec_Tutte_UIHPM}.

\begin{proposition}\label{prop_convergence_randomized_perimeters}
We have the local convergence 
\[ M^{\bullet \bullet, z} \xrightarrow[z \to 2]{(d)} M_{\infty}.\]
\end{proposition}

Note that the perimeter of $M^{\bullet \bullet, z}$ goes to infinity as $z \to +\infty$ and that, conditionally on its perimeter, the map $M^{\bullet \bullet, z}$ has the same distribution as $M^{\bullet\bullet}_p$. Therefore, Proposition \ref{prop_convergence_randomized_perimeters} is a weaker version of Theorem \ref{thm_defn_UIHPM}. On the other hand, it implies that if $M^{\bullet\bullet}_p$ converges in distribution, the limit has to be $M_{\infty}$.

The proof consists basically of taking the limit of Proposition \ref{Prop_Decomposition_of_Boltzmann_quad_is_Boltzmann} as $z \to 2/9$: it turns out that the limit of $A^z$ is the \UIHPQ $Q_{\infty}$, and that the limit of its Tutte core (which has the same distribution as $M^{\bullet \bullet, z}$) is the Tutte core of the \UIHPQ. Here are two slight issues that we need to overcome: first, the model $A^z$ is not exactly a uniform quadrangulation (recall that $\triQA \varsubsetneq \triQ$), so we need to check that $A^z$ is close to $Q^{\bullet \bullet, z}$ in order to have convergence to the \UIHPQ. This is done in Lemma \ref{lem_convergence_A_UIHPQ}. 
Second, the Tutte core operation is not continuous for the local topology in general\footnote{For example, if $q_n \to q$ but the distance between the root and the core of $q_n$ goes to $\infty$, then the core of $q_n$ may not converge to the core of $q$ (if it converges at all).}
. The partial continuity result that we use is Lemma \ref{lem_continuity_core}, which is completely deterministic.

\begin{lemma}\label{lem_convergence_A_UIHPQ}
We have the local convergence
\[A^z \xrightarrow[z \to 2/9]{(d)} Q_{\infty}.\]
\end{lemma}

\begin{proof}
We know that the perimeter of $Q^{\bullet \bullet, z}$ goes to $+\infty$ as $z \to \frac{2}{9}$ and that, conditionally on its perimeter, the law of $Q^{\bullet \bullet, z}$ is that of $Q_p$. Therefore, and since $Q_p \to Q_\infty$, we have $Q^{\bullet \bullet, z} \to Q_{\infty}$ in distribution.
Moreover, we know that $A^z$ is just $Q^{\bullet \bullet, z}$ conditionned on belonging to $\triQA$. Therefore, to finish the proof of the lemma it is enough to prove
\[ \P \l( Q^{\bullet \bullet, z} \in \triQA \r) \xrightarrow[z \to 2/9]{} 1.\]
This quantity is given by
\[ \P \l( Q^{\bullet \bullet, z} \in \triQA \r) = \frac{\triQQA \l( \frac{1}{12}, z \r)}{\triQQ \l( \frac{1}{12}, z \r)}. \]
The equivalent of the denominator as $z \to \frac{2}{9}$ is given by \eqref{computation_Q_critical}. On the other hand, the numerator can be estimated using \eqref{Eq_serie_generatrice_A} together with \eqref{formula_gen_function_J} and \eqref{computation_M_critical}: it is also equivalent to $\frac{2 \sqrt{2}}{27 \sqrt{3}} \frac{1}{\sqrt{2/9-z}}$, which concludes the proof of the lemma.
\end{proof}

For the next result, we will use the following notation: if $q_n$ is a finite quadrangulation, we denote by $\widetilde{q}_n$ the part of $q_n$ lying between the root and the Tutte core of $q_n$. More precisely, $\widetilde{q}_n$ is the submap of $q_n$ formed by the faces $f$ such that the edge of $\Ti{q}$ inside of $f$ does not belong to $\Core(q_n)$, and can be linked to the root outside of $\Core(q_n)$.

\begin{lemma}\label{lem_continuity_core}
Let $(q_n)$ be a sequence of finite quadrangulations with a simple boundary converging for the local topology to an infinite half-plane quadrangulation $q$. We assume that
\begin{enumerate}
\item[(i)]\label{item_uniqueness_core}
the map $\Ti{q}$ has a unique infinite bead $\Core(q)$;
\item[(ii)]\label{item_core_gets_larger}
the size of $\Core(q_n)$ goes to infinity;
\item[(iii)]\label{item_bounded_bead}
$(\widetilde{q}_n)$ is constant and finite for $n$ large enough.
\end{enumerate}
Then we have the local convergence
\[ \Core(q_n) \xrightarrow[n \to +\infty]{} \Core(q).\]
\end{lemma}

\begin{proof}
We first note that if $u,v$ are two white vertices of $q$, then $d_{\Ti{q}}(u,v) \geq (1/2) d_q(u,v)$. Hence, recalling from Section \ref{subsec_basics} the definition of the “ball” $[m]_r$, it is straightforward to check that $[\Ti{q}]_r$ is a function of $[q]_{2r+2}$. By the definition of the local topology, this implies that $\Ti{q_n} \to \Ti{q}$.

Now let $\wt m_n$ be the connected component of the root in $\Ti{q_n} \backslash \Core(q_n)$, i.e. $\wt m_n$ is the part of $\Ti{q_n}$ corresponding to $\wt q_n$ via the Tutte mapping. Since $\wt q_n$ (and thus $m_n$) is constant and finite for $n$ large enough, let $D$ be the eventual distance in $\Ti{q_n}$ between the root and $\Core(q_n)$. Then $\left[ \Core(q_n)\right]_r$ is obtained from $\left[ \Ti{q_n} \right]_{r+D}$ by removing $\wt m_n$, as well as all the beads which do not touch $\wt m_n$ (because they correspond to the other finite beads of $\Ti{q_n}$). In particular $\left[ \Core(q_n)\right]_r$ is constant eventually, so $\Core(q_n)$ converges locally.

We now need to prove that the limit of $\Core(q_n)$ is $\Core(q)$. This is equivalent to showing $\wt m'=\wt m$ for $n$ large enough, where $\wt m'$ is the eventual value of $\wt m_n$ (which exists by Item (iii)), and $\wt m$ is the connected component of the root in $\Ti{q} \backslash \Core(q)$. Let $r$ be larger than the diameters of $\wt m$ and $\wt m'$. We know that $\wt m' \subset \Ti{q_n}$ for $n$ large enough. By Item (ii), we know that $\Core(q_n)$ contains vertices at distance $r$ from the root of $q_n$. Hence $\wt m_n=\wt m'$ is the largest part of $\Ti{q_n}$ which is separated by a pinchpoint from all the points at distance $r$ from the root. Since $\left[ \Ti{q_n} \right]_r=\left[ \Ti{q} \right]_r$ for $n$ large enough, we deduce that $\wt m'$ is the largest part of $\Ti{q}$ which is separated by a pinchpoint from all the points at distance $r$ from the root. Since $r$ is larger than the diameter of $\wt m$, this is also true for $\wt m$, so $\wt m'=\wt m$, which proves the lemma.
\end{proof}

\begin{proof}[Proof of Proposition \ref{prop_convergence_randomized_perimeters}]
For every $k$, let $z_k \in \l[ 0,2 \r)$ with $z_k \to 2$. Let also $\widetilde{z}_k$ be such that $\frac{1}{\widetilde{z}_k} \monoPP \l( \frac{1}{12}, \widetilde{z}_k \r)=z_k$, so that $\widetilde{z}_k \to 2/9$. By Lemma \ref{lem_convergence_A_UIHPQ}, we know that $A^{\widetilde{z}_k} \to Q_{\infty}$ in distribution. Moreover, for every $k$, let $C_k$ be the component of $A^{\widetilde{z}_k}$ which separates the root of $A^{\widetilde{z}_k}$ from its Tutte core. By Proposition \ref{Prop_Decomposition_of_Boltzmann_quad_is_Boltzmann}, we know that $C_k$ converges in distribution to the finite map $J^{\bullet, 2/9}$. This implies that the pair $\l( A^{\widetilde{z}_k}, C_k \r)$ is tight. Hence, up to extracting a subsequence from $(z_k)$, we may assume the joint convergence
\begin{equation}\label{eqn_joint_convergence_A_C}
\l( A^{\widetilde{z}_k}, C_k \r) \xrightarrow[k \to +\infty]{} \l( Q'_{\infty}, C \r)
\end{equation}
in distribution, for a pair $(Q'_\infty, C)$ of random variables such that $Q'_\infty$ has the law of $Q_\infty$ and $C$ has the law of $J^{\bullet, 2/9}$ (in particular it is a.s. finite). Moreover, we know by Proposition \ref{Prop_Decomposition_of_Boltzmann_quad_is_Boltzmann} that the core of $A^{\widetilde{z}_k}$ has the law of $M^{\bullet \bullet, z_k}$, and that $|M^{\bullet \bullet, z_k}| \to +\infty$ in probability (because the generating function blows up), so $\left| \Core(A^{\widetilde{z}_k}) \right| \to +\infty$ jointly with \eqref{eqn_joint_convergence_A_C}. By the Skorokhod embedding theorem, we may assume these convergences are almost sure. In particular, the sequence $\left( C_k \right)$ is a.s. ultimately constant. Hence, the sequence of quadrangulations $\l( A^{\widetilde{z}_k} \r)_{k \geq 1}$ almost surely satisfies all the assumptions of Lemma \ref{lem_continuity_core}. Therefore, we have
\[ M^{\bullet \bullet, z_k} = \Core \l( A^{\widetilde{z}_k} \r) \xrightarrow[k \to +\infty]{a.s.} \Core (Q_{\infty})=M_{\infty},\]
so the convergence holds in distribution. We have just proved that every sequence $(z_k) \to 2$ has a subsequence along which $M^{\bullet \bullet, z}$ converges in distribution to $M_{\infty}$, which proves that $M^{\bullet \bullet, z}$ converges to $M_{\infty}$ as $z \to 2$.
\end{proof}

\subsection{Convergence of Boltzmann maps with fixed perimeter}

We now finish the proof of Theorem \ref{thm_defn_UIHPM}, i.e. of convergence of $M_p$ to $M_{\infty}$. The basic idea is to prove, using classical ideas and the asymptotics \eqref{asymptotic coefficients_M}, that the law of $M_p$ converges in a weak sense, and then to use Proposition \ref{prop_convergence_randomized_perimeters} to identify the limit as $M_{\infty}$.

For this, we introduce a notion of sub-map of a half-plane map, just like the sub-maps of quadrangulations defined in Section \ref{asymptotic coefficients_M}. Let $m$ be a map with a simple boundary, with a marked segment on the boundary containing the root edge of $m$. If $M$ is a half-plane map with a simple boundary, we write $m \subset M$ if there is a neighbourhood $V$ of the root of $M$ isomorphic to $m$, such that the roots are identified and the marked segment of $\partial m$ is identified with $\partial V \cap \partial M$.

\begin{lemma}\label{lem_limit_of_proba_submap}
For any possible sub-map $m$, the probability
\[ \P \l( m \subset M_p \r)\]
has a limit as $p \to +\infty$.
\end{lemma}

\begin{proof}
We denote by $\partial_b m$ the \emph{bottom boundary} of $m$ (i.e. its marked segment), and by $\partial_t m$ its \emph{top boundary} (i.e. the rest of $\partial m$). Let $p > |\partial_b m|$. 
If $m \subset M_p$, then the complementary of $m$ is a planar map with a simple boundary of length
\[p-|\partial_b m|+|\partial_t m|.\]
By summing over every possible value of this complementary, we obtain
\[ \P \l( m \subset M_p \r) = \l( \frac{1}{12} \r)^{|m \backslash \partial_t m|} \frac{\monoMM_{p-|\partial_b m|+|\partial_t m|}}{\monoMM_p}, \]
where $|m \backslash \partial_t m|$ is the number of edges of $m$ which do not belong to its top boundary. The convergence of this quantity as $p \to +\infty$ follows from the asymptotics \eqref{asymptotic coefficients_M}.
\end{proof}

\begin{proof}[End of the proof of Theorem \ref{thm_defn_UIHPM}]
Since a map is characterized by the set of its finite sub-maps, it is enough to prove that, for any possible sub-map, we have
\[ \P \l( m \subset M_p \r) \xrightarrow[p \to +\infty]{} \P \l( m \subset M_{\infty} \r).\]
Hence, we fix such a map $m$, and denote by $\ell_m$ the limit given by Lemma \ref{lem_limit_of_proba_submap}.

By definition, conditionally on its perimeter, the map $M^{\bullet \bullet, z}$ has the same law as $M_p$, where two additional uniform boundary edges (conditionned to be distinct from each other and from the root) have been marked. Therefore, if we forget the two additional marked edges, then $M^{\bullet \bullet, z}$ has the same law as $M_{P_z}$, where $P_z$ is a certain random variable on $\N$. Moreover, as we have seen above (see the discussion right after Proposition \ref{Prop_Decomposition_of_Boltzmann_quad_is_Boltzmann}), the perimeter $P_z$ goes to $+\infty$ in probability as $z \to 2$. Therefore, we have
\[ \P \l( m \subset M^{\bullet \bullet, z} \r) \xrightarrow[z \to 2]{} \ell_m.\]
On the other hand, by Proposition \ref{prop_convergence_randomized_perimeters}, this must also converge to $\P \l( m \subset M_{\infty} \r)$, so $\P \l( m \subset M_{\infty} \r)=\ell_m$, which concludes the proof.
\end{proof}

\section{Perspectives}\label{Sec_perspectives}

The self-duality of uniform planar maps (as opposed to e.g. $d$-angulations) plays a very important role in our proofs, so we do not believe that our method can be extended to more general map models like the \UIPT. Let us simply note that via the \Tutte, we can deduce lower bounds on resistances in the dual map of the \UIHPQ. However, this is of less interest since the dual of a quandrangulation is a bounded-degree graph, so logarithmic lower bounds on the resistances can be obtained via circle packings \cite{BS01} (at least in the full-plane topology). Similarly, upper bounds for resistances in uniform planar maps would imply upper bounds for resistances in uniform quadrangulations.

On the other hand, we hope that self-duality could be used to obtain upper bounds on the resistances in uniform planar maps. The reason why obtaining upper bounds seems more difficult than lower bounds is that this would require to control the resistances at all scales from the root to infinity, whereas lower bounds only required to prove that a “positive proportion” of the scales contribute significantly. A natural first step in this direction would be to prove that the resistance $R_{p_1, p_2, p_3, p_4}$ of Lemma \ref{lem_duality_random} is roughly of order $1$ with large probability (i.e. both the resistance and its inverse are tight) when $p_1, p_2, p_3, p_4$ are of the same order. The next natural step would be to “glue” the different scales together by using Russo--Seymour--Welsh-like ideas.

Another natural direction would be to try to extend our ideas to the full-plane \UIPM. This seems also harder to handle than the half-plane case, because in a peeling exploration of a full-plane map, it is impossible that a single peeling step creates a block separating the root from infinity. Here again, it would therefore be necessary to “glue” different blocks like in the RSW theory.

Finally, while uniform maps seem to be the only natural self-dual model of maps “without matter”, other self-dual models of infinite random planar maps equipped with statistical physics models have already been considered, see e.g. \cite{Chen15}. It is natural to hope that the techniques developed in the present work may be used to say something about the simple random walk on these models.

\bibliographystyle{abbrv}
\bibliography{bibli}

\end{document}